\documentclass{amsart}
\usepackage{hyperref,amssymb}
\usepackage[french, english]{babel}
\usepackage{fancyhdr}
\usepackage{colortbl}  
\usepackage{enumerate} 
\usepackage{graphicx}
\usepackage{float}  
\usepackage{mathrsfs}
\usepackage[left=2.5cm,right=2.5cm,top=2.75cm,bottom=2.75cm]{geometry}
\usepackage{tikz}
\usetikzlibrary{arrows,automata,shapes,calc,patterns}
\usepackage{cite}
\usepackage{xcolor}

\usepackage[colorinlistoftodos]{todonotes}
\definecolor{blue}{rgb}{0,0,1}

\newtheorem{theorem}{Theorem}[section]

\newtheorem{lemma}[theorem]{Lemma}
\newtheorem{cor}[theorem]{Corollary}

\theoremstyle{definition}
\newtheorem{definition}{Definition}[section]

\theoremstyle{remark}
\newtheorem{remark}{Remark}[section]

\numberwithin{equation}{section}

\def \N {\mathbb{N}}  % naturales
\def \R {\mathbb{R}}  % reales
\def\RN {\mathbb{R}^{N}} %R^{N}

\def \Wop {W_0^{1,p}(\omega)}

\def \HH {H_0^1(\omega)}

\def \HLinfty {H_0^1(\Omega) \cap L^{\infty}(\Omega)}

\def \HLLinfty {H_0^1(\omega) \cap L^{\infty}(\omega)}

\def \pLaplac {-\Delta_p} 
\def \gradu {\nabla u}
\def \gradv {\nabla v}

\def \vd {v_{\delta}}

\def \Od {\omega_{\delta}}
\def \Sd {S_{\delta}}

\newcommand{\supp}{\operatorname{Supp}}
\renewcommand{\epsilon} {\varepsilon}

\newcommand{\esssup}{\operatorname{esssup}}

\newcommand{\diam}{\operatorname{diam}}

\hypersetup{
colorlinks=true, breaklinks=true, bookmarks=true,bookmarksnumbered,
urlcolor=green!50!black, linkcolor=blue!85!black, citecolor=red!85!black, % Link colors
pdftitle={}, % PDF title
pdfauthor={\textcopyright}, % PDF Author
pdfsubject={}, % PDF Subject
pdfkeywords={}, % PDF Keywords
pdfcreator={pdfLaTeX}, % PDF Creator
pdfproducer={LaTeX with hyperref and ClassicThesis} % PDF producer
}

\begin{document}

\title[An indefinite elliptic problem with critical growth in the gradient]{A priori bounds and multiplicity of solutions for an indefinite elliptic problem with critical growth in the gradient}

%    Remove any unused author tags.

%    author one information
\author[Colette De Coster, Antonio J. Fern\'andez and Louis Jeanjean]{}
\address{}
\curraddr{}
\email{}
\thanks{}

%    author two information
%\author[A. J. Fern\'andez]{Antonio J. Fern\'andez}
%\address{Univ. Valenciennes, EA 4015 - LAMAV - FR CNRS 2956, F-59313 Valenciennes, France}
%\address{Laboratoire de Math\'ematiques (UMR 6623), Universit\'e de Bourgogne Franche-Comt\'e, 
%16 route de Gray, 25030 Besan\c con Cedex, France}
%\curraddr{}
%\email{}
%\thanks{}

%\email{colette.decoster@univ-valenciennes.fr}
%\email{antonio\_jesus.fernandez\_sanchez@univ-fcomte.fr}
%\email{louis.jeanjean@univ-fcomte.fr}


%\keywords{Indefinite problem, quasilinear elliptic equations, critical growth in the gradient, a priori bound, continuum of solutions, p-Laplacian, Harnack inequality}

\date{}

\dedicatory{}
\maketitle

\centerline{\scshape Colette De Coster}
\smallskip
{\footnotesize
 % please put the address of the second  and third author
 \centerline{Univ. Valenciennes, EA 4015 - LAMAV - FR CNRS 2956, F-59313 Valenciennes, France}
\vspace{0.05cm}
\centerline{\textit{E-Mail address} : \texttt{colette.decoster@uphf.fr}}
}

\bigskip

\centerline{\scshape Antonio J. Fern\'andez}
\smallskip
{\footnotesize
% please put the address of the first author
 \centerline{Univ. Valenciennes, EA 4015 - LAMAV - FR CNRS 2956, F-59313 Valenciennes, France}
 \vspace{0.09cm}
   \centerline{Laboratoire de Math\'ematiques (UMR 6623), Universit\'e de Bourgogne-Franche-Comt\'e,}
 \centerline{16 route de Gray, 25030 Besan\c con Cedex, France}
 \vspace{0.05cm}
\centerline{\textit{E-Mail address} : \texttt{antonio\_jesus.fernandez\_sanchez@univ-fcomte.fr}}
} % Do not forget to end the {\footnotesize by the sign }

\bigskip

\centerline{\scshape Louis Jeanjean\footnote{Corresponding author} \footnote{This work has been carried out in the framework of the project NONLOCAL (ANR-14-CE25-0013) funded by the French National Research Agency (ANR).} }
\smallskip
{\footnotesize
% please put the address of the first author
   \centerline{Laboratoire de Math\'ematiques (UMR 6623), Universit\'e de Bourgogne-Franche-Comt\'e,}
 \centerline{16 route de Gray, 25030 Besan\c con Cedex, France}
 \vspace{0.05cm}
\centerline{\textit{E-Mail address}: \texttt{louis.jeanjean@univ-fcomte.fr}}
} % Do not forget to end the {\footnotesize by the sign }

\selectlanguage{english}

\begin{center}\rule{1\textwidth}{0.1mm} \end{center}
\begin{abstract}
%We consider the boundary value problem
Let $\Omega \subset \RN$, $N \geq 2$, be a smooth bounded domain. We consider a boundary value problem of the form
$$-\Delta u = c_{\lambda}(x) u + \mu(x) |\gradu|^2 + h(x)\,, \quad u \in \HLinfty\,$$
where $c_{\lambda}$ depends on a parameter $\lambda \in \R$, the coefficients $c_{\lambda}$  and $h$ belong to  $L^q(\Omega)$ with $q>N/2$ and $\mu \in L^{\infty}(\Omega)$.  Under suitable assumptions, but without imposing a sign condition on any of these coefficients, we obtain an a priori upper bound on the solutions. Our proof relies on a new boundary weak Harnack inequality. This inequality, which is of independent interest, is established in the general framework of the $p$-Laplacian. 
With this a priori bound at hand, we show the existence and multiplicity of solutions.
\end{abstract}

\begin{center} \rule{1 \textwidth}{0.1mm} \end{center}

\selectlanguage{french}
\begin{abstract}
%We consider the boundary value problem
Soit $\Omega \subset \RN$, $N \geq 2$, un domaine born\'e r\'egulier. Nous consid\'erons un probl\`eme aux limites de la forme
$$-\Delta u = c_{\lambda}(x) u + \mu(x) |\gradu|^2 + h(x)\,, \quad u \in \HLinfty\,$$
o\`u $c_{\lambda}$ d\'epend d'un param\`etre $\lambda \in \R$, les coefficients $c_{\lambda}$  et $h$ sont des fonctions dans $L^q(\Omega)$  avec $q>N/2$ et $\mu \in L^{\infty}(\Omega)$.  Sous certaines hypoth\`eses, mais sans imposer une condition de signe sur aucun des coefficients, nous obtenons une borne \`a priori sup\'erieure sur les solutions. Notre preuve repose sur une nouvelle in\'egalit\'e de Harnack au bord. Cette in\'egalit\'e, qui est d'int\'er\^et propre, est \'etablie dans le cadre plus g\'en\'eral du $p$-Laplacien. L'obtention d'une borne \`a priori nous permet de d\'emontrer l'existence et la multiplicit\'e de solutions.
\end{abstract}

\begin{center} \rule{1 \textwidth}{0.1mm} \end{center}

\selectlanguage{english}

\section{Introduction and main results} \label{I}

The paper deals with the existence and multiplicity of solutions for boundary value problems of the form
\[ \label{Qlambda} \tag{$Q_{\lambda}$}
-\Delta u =  c_{\lambda}(x)  u + \mu(x) |\gradu|^2 + h(x)\,, \quad u \in \HLinfty\,,\]
with $ c_{\lambda}$ depending on a real parameter $\lambda$. Here $\Omega \subset \R^N$, $ N \geq 2$, is a bounded domain with boundary $\partial \Omega$ of class $\mathcal{C}^{1,1}$, $c_{\lambda}$  and $h$ belong to $L^q(\Omega)$ for some $q > N/2$ and $\mu$ belongs to $L^{\infty}(\Omega)$.
\medbreak

This type of problem, which started to be studied by L. Boccardo, F. Murat and J.P. Puel in the 80's, has attracted a new attention these last years.
Under the condition $c_{\lambda} \leq -\alpha_0 < 0$ a.e. in $\Omega$ for some $\alpha_0 > 0$,
the existence of a solution of \eqref{Qlambda} is a particular case of the results of \cite{B_M_P_1983, B_M_P_1992} and its  
uniqueness follows from \cite{B_B_G_K_1999, B_M_1995}. The   case $c_{\lambda} \equiv 0$  
was studied in \cite{F_M_2000, A_DA_P_2006} and the existence requires some smallness condition on $\| \mu h\|_{N/2}$. 
The situation where one only requires $c_{\lambda} \leq  0$   a.e. in $\Omega$ (i.e. allowing parts of 
the domain where $c_{\lambda} \equiv 0$   and parts of it where $c_{\lambda} < 0$  )
proved to be more complex to treat. In the recent papers \cite{A_DC_J_T_2015, DC_F_2018}, the authors explicit sufficient conditions  for the existence of solutions of \eqref{Qlambda}. Moreover, in \cite{A_DC_J_T_2015}, the uniqueness of solution is established (see also \cite{A_DC_J_T_2014} in that direction). All these results were obtained without requiring any sign conditions on $\mu$ and $h$.
\medbreak

In case $c_{\lambda} = \lambda c \gneqq 0$, as we shall discuss later, problem \eqref{Qlambda} behaves very differently and becomes much richer.  Following \cite{S_2010}, which considers a particular case,
%{\color{red}Rajouter dans la bibliographie [
%B. Sirakov, Solvability of uniformly elliptic fully nonlinear PDE, Arch. Ration. Mech. Anal. 195 (2010) 579–607. ]} which consider a particular case,
 \cite{J_S_2013} studied \eqref{Qlambda} with $\mu(x) \equiv \mu > 0$ and 
$\lambda c \gneqq 0$ but without a sign condition on $h$. The authors proved the existence of at least two 
solutions when $\lambda>0$ and $\|(\mu h)^{+}\|_{N/2}$ are small enough. 
The restriction $\mu$ constant was removed in \cite{A_DC_J_T_2015} and extended to $\mu(x) \geq \mu_1 > 0$ a.e. in $\Omega$, at the expense of adding the hypothesis $h \gneqq 0$. 
Next, in \cite{DC_J_2017}, assuming stronger regularity on $c$ and $h$,  the authors  removed the condition $h \gneqq 0$. In this paper, it is also lightened that the structure of the set of solutions when $\lambda >0$, crucially depends  on the sign of the (unique) solution of $(Q_0)$.  Note that, in \cite{DC_F_2018}, the above results are extended to the $p$-Laplacian case. Also,  in the frame of viscosity solutions and fully nonlinear equations, under corresponding assumptions, similar conclusions have been obtained very recently in \cite{SN_S_2018}.
\medbreak

We refer to \cite{J_S_2013} for an heuristic explanation on how the behavior of \eqref{Qlambda} is affected by the change of sign in front of the linear term. Actually, in the case where $\mu(x) \equiv \mu$ is a constant, it is possible to transform problem \eqref{Qlambda} into a new one which admits a variational formulation. 
When $c_{\lambda} \leq - \alpha_0 < 0$, the associated functional, defined on $H_0^1(\Omega)$, is coercive. If $c_{\lambda} \lneqq 0$, the coerciveness may be lost and when $c_{\lambda} \gneqq 0$, in fact as soon as $c_{\lambda}^+ \gneqq 0$, the functional is unbounded from below.
In \cite{J_S_2013} this variational formulation was directly used to obtain the solutions. In \cite{A_DC_J_T_2015, DC_J_2017} where $\mu$ is non constant, topological arguments, relying on the derivation of a priori bounds for certain classes of solutions, were used.
\medbreak

The only known results where $c_{\lambda}$ may change sign are \cite{J_RQ_2016, DC_F_2018-A3} (see also \cite{GM_I_RQ_2015} for related problems). %{\color{red}Ajouter notre article 3}.
They both concern the case where $\mu$ is 
a positive constant. In \cite{J_RQ_2016}, assuming $h \gneqq 0$,   $\mu h$ and $c_{\lambda}^+$ small in an appropriate sense, the existence of at least two non-negative solutions was proved. In \cite{DC_F_2018-A3}, the authors show that the loss of positivity of the coefficient of $u$ does not affect the structure of the set of  solutions of \eqref{Qlambda} observed in \cite{DC_J_2017} when
$c_{\lambda}=\lambda c \gneqq 0$.  Since $\mu$ is constant in \cite{J_RQ_2016, DC_F_2018-A3}, it is possible to treat the problem variationally.  The main issue, to derive the existence of solutions, is then to show the boundedness of the Palais-Smale sequences.   
\medbreak

When $c_{\lambda} \gneqq 0$, all the above mentioned results require either $\mu$ to be constant or to be uniformly bounded from below by a positive constant (or similarly bounded from above by a negative constant). In \cite{Soup_2015}, assuming that the three coefficients functions are non-negative, a first attempt to remove these restrictions on $\mu$ is presented.
Following the approach of \cite{A_DC_J_T_2015}, the proofs of the existence results reduce to obtaining a priori bounds on the non negative solutions  of \eqref{Qlambda}. First it is observed in \cite{Soup_2015} that a necessary condition is the existence of a ball $B(x_0,\rho)\subset\Omega$  and 
$\nu>0$ such that $\mu\geq \nu$ and $c\geq \nu$ on  $B(x_0,\rho)$. When $N=2$ this condition also  proves to be sufficient. If $N=3$ or $4$ the condition $\mu\geq \mu_0>0$ on a set
$\omega \subset \Omega$ such that $\mbox{supp}(c)\subset\overline{\omega}$ permits to obtain the a priori bounds. Other sets of conditions are presented when $N=3$ and $N=5$. However, if the approach developed in \cite{Soup_2015}, which relies on interpolation and elliptic estimates in weighted Lebesgue spaces, works well in low dimension, the possibility to extend it to dimension $N \geq6$ is not apparent.  

\medbreak

In this paper we pursue the study of \eqref{Qlambda} and consider situations where the three coefficients functions $c_{\lambda}$,  $\mu$ and $h$ may change sign. We define for $v \in L^1(\Omega)$, $v^+= \max(v,0)$ and $v^- =
\max(-v,0)$. 
As observed already in \cite{DC_F_2018-A3}, the structure of the solution set depends on the size of the  
positive hump (i.e. $c_{\lambda}^+$) but it is not affect by the size of the  negative hump (i.e. $c_{\lambda}^-$). 
Hoping to clarify this,  we now write $c_{\lambda}$ under the form $c_{\lambda}=\lambda c_+- c_-$
and
%Hoping to clarify further the role of the positive part, we modify the formulation of our problem and now 
consider the problem
\[ \label{Plambda} \tag{$P_{\lambda}$}
-\Delta u = (\lambda c_+(x)- c_-(x))  u + \mu(x) |\gradu|^2 + h(x)\,, \quad u \in \HLinfty\,,\]
under the assumption
\[ \label{A1} \tag{$A_1$} 
\left\{
\begin{aligned}
&\Omega \subset \RN,\, N \geq 2, \textup{ is a bounded domain with boundary }\partial \Omega \textup{  of class }\mathcal{C}^{1,1}, 
\\&c_+, c_- ,  h^{+} \in L^q(\Omega) \textup{ for some } q > N/2 \,,   
\  \mu,  h^{-} \in L^{\infty}(\Omega) \,, 
\\
& c_+(x) \geq 0, \ c_-(x) \geq 0 \textup{ and } c_-(x) c_+(x) =0 \textup{ a.e. in }\Omega,
\\
& |\Omega_{+}|> 0, \textup{ where } \Omega_{+} :=  \supp(c_{+})
%\textup{ is  an open subset of } \Omega \textup{ such that } 
%%\textup{  where } c > 0 \textup{  a.e. },
%  % \textup{ and } 
% |\supp(c^{+}) \setminus \Omega_{+}| = 0
%  \textup{ and }  c^+=0  \textup{ a.e. in } \Omega\setminus\Omega_+,
  \\
&\textup{there exists a }\epsilon>0 \textup{ such that  }\mu(x) \geq \mu_1 > 0 \textup{ and }c_- = 0 \textup{ in } \{x\in \Omega : d(x,\Omega_+)<\epsilon\}.
%\textup{ of } \overline{\Omega}_{+}  \textup{ with }\partial \Omega_1 
% \textup{  of class }\mathcal{C}^{1,1}.
\end{aligned}
\right.
\] 
For a definition of $\supp(f) $ with $f \in L^p(\Omega)$, for some $p \geq 1$, we refer to \cite[Proposition 4.17]{Brezis}.  Note also that the condition that $c_-=0$ on $\{x\in \Omega : d(x,\Omega_+)<\epsilon\}$ for some $\epsilon >0$, is reminiscent of the so-called ``thick zero set" condition first introduced in \cite{AlTa}.
\medbreak

We also observe that, under the regularity assumptions of condition \eqref{A1}, any solution of  \eqref{Plambda} belongs to $\mathcal{C}^{0,\tau}(\overline{\Omega})$ for some $\tau > 0$. This can be deduce from \cite[Theorem IX-2.2]{L_U_1968}, see also
\cite[Proposition 2.1]{A_DC_J_T_2014}.
\medbreak

As in \cite{A_DC_J_T_2015,DC_J_2017, Soup_2015} we obtain our results using a topological approach, relying thus on the derivation of a priori bounds.  In that direction our main result is the following.
\begin{theorem} 
\label{aPrioriBound} 
Assume \eqref{A1}.
% and suppose that \eqref{P0} has a non-negative solution. 
Then, for any $ \Lambda_2 > \Lambda_1 > 0$, there exists a constant $M > 0$ such that, for each 
$\lambda \in [\Lambda_1, \Lambda_2]$, any 
%non-negative 
solution of \eqref{Plambda} satisfies 
$\sup_{\Omega} u  \leq M$.
\end{theorem}

Having at hand this a priori bound, following the strategy of \cite{A_DC_J_T_2015}, we show the existence of a continuum of solutions of \eqref{Plambda}. More precisely, defining
\begin{equation} \label{sigma}
\Sigma := \{ (\lambda, u ) \in \R \times \mathcal{C}(\overline{\Omega})  :
 u \textup{ solves } \eqref{Plambda} \},
\end{equation} 
we prove the following theorem.

\begin{theorem} \label{th1}
Assume \eqref{A1} and suppose that $(P_0)$ has a solution $u_0$ with $c_+u_0 \gneqq 0$. 
Then, there exists a continuum $\mathscr{C} \subset \Sigma$ such that the projection of $\mathscr{C}$ on the 
$\lambda$-axis is an unbounded interval $(-\infty,\overline{\lambda}]$ for some 
$\overline{\lambda} \in (0,+\infty)$ and $\mathscr{C}$ bifurcates from infinity to the right of the axis $\lambda = 0$. 
Moreover:
\begin{itemize}
\item[1)]
for all $\lambda\leq0$, the problem \eqref{Plambda} has an unique solution $u_{\lambda}$   and this solution satisfies $u_0-\|u_0\|_{\infty}\leq u_{\lambda}\leq u_0$.
\item[2)]
 there exists $\lambda_0 \in (0, \overline{\lambda}]$ such that, for all $\lambda \in (0,\lambda_0)$, 
 the problem \eqref{Plambda} has at least two solutions with $u_i\geq u_0$ for $i=1$, $2$.
\end{itemize}
\end{theorem}

\begin{remark} \mbox{}

\begin{itemize}
\item[(a)]
 Theorem \ref{th1}, 1) generalizes   \cite[Theorem 1.2]{A_DC_J_T_2015}.
% that we give here   for completeness. %
\medbreak

\item[(b)]
Note that problem $(P_0)$ is given by
$$-\Delta u =  - c_{-}(x)u + \mu(x) |\gradu|^2 + h(x)\,, \qquad u \in \HLinfty\,.$$
In \cite{A_DC_J_T_2015, DC_F_2018} the authors give sufficient conditions to ensure the existence of a
solution of $(P_0)$. 
Moreover, if $h \geq 0$ in $\Omega$, \cite[Lemma 2.2]{A_DC_J_T_2014} implies that the solution of $(P_0)$ is non-negative.
\end{itemize}
\end{remark}

Let us give some ideas of the proofs. As we do not have global sign conditions, the approaches used in \cite{A_DC_J_T_2015,DC_J_2017,Soup_2015} to 
obtain the a priori bounds do not apply anymore and another strategy is required. 
To this aim, we further develop some techniques first sketched in the unpublished work
\cite{S_2015}. These techniques, in the framework of viscosity solutions of fully nonlinear equations, 
now lies at the heart of the paper  \cite{SN_S_2018}. We also make use of some ideas 
from \cite{GM_I_RQ_2015}. First we show, in Lemma \ref{Step 1}, that it is sufficient to control the behavior
 of the solutions on $\overline{\Omega}_{+}$. By compactness, we are then reduced to study what happens around an (unknown) point $\overline{x} \in \overline{\Omega}_{+}$. We shall consider separately the alternative cases $\overline{x} \in \overline{\Omega}_{+} \cap \Omega$ and $\overline{x} \in \overline{\Omega}_{+} \cap \partial \Omega$. A local analysis is made respectively in a ball or a semiball centered at $\overline{x}$.
%By compactness we are then reduce to 
%study what happens on a ball centered in a point $\overline{x}$ of $\overline\Omega_+ 
%\cap \Omega$ as well as on a semi ball centered in a point $\overline{x}$
%of $\overline\Omega_+ \cap \partial\Omega$.
If similar analysis, based on the use of Harnack type 
 inequalities, had previously been performed in other contexts when $\overline{x} \in \Omega$,  
 we believe it is not the case when $\overline{x} \in \partial \Omega$. 
 For $\overline{x} \in \partial \Omega$, the key to our approach is the use of boundary weak Harnack inequality.
 Actually a major part of the paper is devoted to establishing  this inequality. 
 This is done in  a more general context than needed for \eqref{Plambda}. 
 In particular it also cover the case of the $p$-Laplacian with a zero order term. We  believe
 that this ``boundary weak Harnack inequality", see Theorem  \ref{BWHIP}, 
 is of independent interest and will  proved to be useful in other settings. Its proof uses ideas  introduced by 
%B. Sirakov  \cite{S_2015}, see also \cite{S_2017} 
  B. Sirakov  \cite{S_2017}. In \cite{S_2017} such type of inequalities is 
established for an uniformly elliptic operator and viscosity solutions. However, since our context is 
quite different, the result of \cite{S_2017}  does not apply to our situation and we need
 to work out an adapted proof.
\medbreak

We now describe the organization of the paper. In Section \ref{II}, we present some preliminary results which are needed in the development of our proofs. In Section \ref{appBWHI}, we prove the boundary weak Harnack inequality for the $p$-Laplacian. The a priori bound, namely Theorem \ref{aPrioriBound}, is proved in Section \ref{III}. Finally Section \ref{IV} is devoted to the proof of Theorem \ref{th1}.
\bigbreak

\noindent \textbf{Notation.}
\begin{enumerate}{\small 
\item[1)] In  $\mathbb R^N$, we use the notations $|x|=\sqrt{x_1^2+\ldots+x_N^2}$ and $B_R(y)=\{x\in \mathbb R^N : |x-y|<R\}$.
%\item[2)] For $v \in L^1(\Omega)$ we define $v^+= \max(v,0)$ and $v^- =
%\max(-v,0)$.
\item[2)] We denote $\mathbb R^+=(0,+\infty)$,  $\mathbb R^-=(-\infty,0)$ and $\mathbb N=\{1,2,3,\ldots\}$.
\item[3)] For $h_1$, $h_2\in L^1(\Omega)$ we write
\begin{itemize}
\item $h_1\leq h_2$ if $h_1(x)\leq h_2(x)$ for a.e. $x\in\Omega$,
\item $h_1\lneqq h_2$ if $h_1\leq h_2$ and 
$\textup{meas}(\{x\in\Omega:
h_1(x)<h_2(x)\})>0$.
\end{itemize}}
\end{enumerate}

\section{Preliminary results} \label{II}

In this section, we collect some results which will play an important role throughout the work. First of all, let us consider the boundary value problem 
\begin{equation} \label{eqlu}
-\Delta u + H(x,u,\gradu) = f, \qquad u \in \HLLinfty.
\end{equation}
Here $\omega \subset \R^N$ is a bounded domain, $f \in L^1(\omega)$ and $H: \omega \times \R \times \RN \rightarrow \R$ is a Carath\'eodory function. 

\begin{definition}\label{lower-upper}
We say that $\alpha \in H^1(\omega) \cap L^{\infty}(\omega)$ is a \textit{lower solution} of \eqref{eqlu} 
if $\alpha^{+} \in \HH$ and, for all $\varphi \in \HLLinfty$ with $\varphi \geq 0$, we have
\[ \int_{\omega} \nabla \alpha \nabla \varphi\, dx + \int_{\omega} H(x,\alpha, \nabla \alpha) \varphi\, dx 
\leq \int_{\omega} f(x) \varphi\, dx\,.\]
Similarly, $\beta \in H^{1}(\omega) \cap L^{\infty}(\omega)$ is an \textit{upper solution} of \eqref{eqlu} 
if $\beta^{-} \in \HH$ and, for all $\varphi \in H_0^1(\omega) \cap L^{\infty}(\omega)$ with $\varphi \geq 0$, we have
\[ \int_{\omega}  \nabla \beta \nabla \varphi\,dx + \int_{\omega} H(x,\beta, \nabla \beta) \varphi\, dx 
\geq \int_{\omega} f(x) \varphi\,dx\,.\] 
\end{definition}

Next, we consider the boundary value problem
\begin{equation} \label{eqBC}
-\Delta u + a(x)u = b(x)\,, \qquad u \in \HH\,,
\end{equation}
under the assumption
\begin{equation} \label{hypLMP}
\left\{
\begin{aligned}
& \omega \subset \RN,\ N \geq 2\,, \textup{ is a bounded domain,} \\
& a,\,\, b \in L^r(\omega)\, \textup{ for some }\, r > N/2.
\end{aligned}
\right.
\end{equation}

\begin{remark}
With the regularity imposed in the following lemmas and in the absence of a gradient term in the equation, we do not need the lower and upper solutions to be  bounded. The full Definition \ref{lower-upper} will however be needed in other parts of the paper.  
\end{remark}

\begin{lemma} \rm\textbf{(Local Maximum Principle)} \label{LMP}
\it Under the assumption \eqref{hypLMP}, assume that $u \in H^1(\omega)$ is a lower solution of \eqref{eqBC}.
%{eqLMP}. 
For any ball $B_{2R} (y) \subset \omega$ and any $ s > 0$, 
there exists $C = C(s,r,\|a\|_{L^r(B_{2R}(y))},R) > 0$ such that
\[ \sup_{B_R(y)}u^{+} \leq C \Big[ \Big(\int_{B_{2R}(y)} (u^{+})^s dx \Big)^{1/s}
 + \|b^+\|_{L^r(B_{2R}(y))}\Big] \,.\]
\end{lemma}
%{\color{red}Verifier la dependance des constantes en particulier independance par rapport a y}
\begin{proof}
See for instance \cite[Theorem 8.17]{G_T_2001_S_Ed} and  \cite[Corollary 3.10]{M_Z_1997}.
\end{proof}

\begin{lemma} \rm\textbf{(Boundary  Local Maximum Principle)} \label{BLMP}
\it Under the assumption \eqref{hypLMP}, assume that $u \in H^1(\omega)$ is a lower solution of \eqref{eqBC} 
and let $x_0 \in \partial \omega$. For any $R > 0$ and any $ s > 0$, there exists 
$C = C(s,r,\|a\|_{L^r(B_{2R}(x_0) \cap \omega)},R) > 0$ such that
\[ \sup_{B_R(x_0) \cap \omega}u^{+} \leq C \Big[ \Big(\int_{B_{2R}(x_0) \cap \omega} (u^{+})^s dx 
\Big)^{1/s} + \|b^+\|_{L^r(B_{2R}(x_0) \cap \omega)}\Big] \,.\]
\end{lemma}
% {\color{red}Verifier la dependance des constantes en particulier independance par rapport a y}
\begin{proof}
See for instance \cite[Theorem 8.25]{G_T_2001_S_Ed} and \cite[Corollary 3.10 and  Theorem 3.11]{M_Z_1997}.
\end{proof}

\begin{remark}
Lemmas \ref{LMP} and \ref{BLMP} proof's are done in \cite{G_T_2001_S_Ed} for $a \in L^{\infty}(\omega)$ 
and $s>1$. Nevertheless, as it is remarked on page 193 of that book, the proof is valid for
$a \in L^r(\omega)$ with $r > N/2$ and,  following closely the proof of \cite[Corollary 3.10]{M_Z_1997}, the proofs can be extended for any $s > 0$.
\end{remark}

\begin{lemma}\rm\textbf{(Weak Harnack Inequality)} \label{WHI}
\it Under the assumption \eqref{hypLMP}, assume that $u \in H^1(\omega)$ is a non-negative upper solution
 of \eqref{eqBC}. Then, for any ball $B_{4R}(y) \subset \omega$ 
 and any $1 \leq s < \frac{N}{N-2}$ there exists $C = C(s,r,\|a\|_{L^r(B_{4R}(y))},R) > 0$ such that
\[\inf_{B_R(y)} u  \geq C \Big[ \Big( \int_{B_{2R}(y)} u^s dx \Big)^{1/s}
 - \|b^-\|_{L^{r}(B_{4R}(y))} \Big] \,.\]
\end{lemma}
%{\color{red}Pourquoi mettre la dependance en $R$ ici et pas ailleurs?}
%{\color{red}Es-tu sur que les d\'ependence ne sont pas les m\^eme pour les trois lemmes?}
%Verifier l'independance de la constante par rapport a $y$}
\begin{proof}
See for instance \cite[Theorem 8.18]{G_T_2001_S_Ed} and   \cite[Theorem 3.13]{M_Z_1997}.
\end{proof}

Now, inspired by \cite[Lemma 3.2]{B_C_1998} (see also \cite[Appendix A]{D_2011}$\,$), we establish the following version of the Brezis-Cabr\'e Lemma.
%{\color{red}J'ai cru voir aussi quelque chose dans Brezis-Nirenberg, a voir}
\begin{lemma} \label{bcLemma1}
Let $\omega \subset \RN$, $N \geq 2$,  be a bounded domain with boundary $\partial \omega$ of class $\mathcal{C}^{1,1}$ and let $a \in L^{\infty}(\omega)$ and $f \in L^1(\omega)$ be non-negative functions. Assume that $ u \in H^1(\omega)$ is an upper solution of
\[ -\Delta u + a(x) u = f(x)\,, \quad u \in H_0^1(\omega)\,.\]
Then, for every $B_{2R} (y)\subset \omega$, there exists 
$C = C(R,y,\omega,\|a\|_{\infty}) > 0$ such that
\[ \inf_{\omega} \frac{u(x)}{d(x,\partial \omega)} \geq C \int_{B_{R}(y)} f(x) \,dx\,. \]
\end{lemma}

\begin{proof}
First of all, as $a$ and $f$ are non-negative, by the weak maximum principle, it follows that
\[\inf_{\omega} \frac{u(x)}{d(x,\partial \omega)}\geq 0\,.\]
%Also observe that $u \in H^1(\omega)$ is an upper solution to 
%\[ -\Delta u + ||a||_{\infty} u = f(x)\,, \quad u \in H_0^1(\omega)\,.\]
Now  let $B_{2R} (y)\subset \omega\,.$ By the above inequality, we can assume without loss of generality that
\[
\int_{B_R(y)} f(x)\, dx > 0\,.
\]
We split the proof into three steps.
\medbreak
\noindent \textbf{Step 1:} \textit{There exists $c_1 = c_1(R,y,\omega,\|a\|_{\infty}) > 0$ such that}
\begin{equation} \label{bc3}
\frac{u(x)}{d(x,\partial \omega)} \geq c_1 \int_{B_R(y)} f(x) \,dx\,, \quad \forall\ x \in  \overline{B_{R/2}(y)}\,.
\end{equation}
%\medbreak
\indent Since $f$ is non-negative, observe that $u$ is a non-negative upper solution of 
\[ -\Delta u + a(x) u = 0\,, \quad u \in H_0^1(\omega)\,.\]
Hence, by Lemma \ref{WHI}, there exists a constant 
$c_2 = c_2(R,\|a\|_{\infty}) > 0$ such that
\begin{equation} \label{bc1}
u(x) \geq c_2\int_{B_{R}(y)} u\, dx\,, \quad \forall\ x \in \overline{B_{R/2}(y)}\,.
\end{equation}
Now, let us denote by $\xi$ the solution of
\begin{equation}
\left\{
\begin{aligned}
 -\Delta \xi + \|a\|_{\infty}\xi & = \chi_{B_{R}(y)}\,, & \textup{ in } \omega\,,\\
 \xi & = 0\,, & \textup{ on } \partial \omega\,.
\end{aligned}
\right.
\end{equation}
By \cite[Theorem 3]{B_N_1993}, there exists a constant $c_3 = c_3(R,y, \omega,\|a\|_{\infty}) > 0$ such that, for all $ x \in \omega$, $\xi(x) \geq c_3 d(x,\partial \omega)$. Thus, since $B_{2R}(y) \subset \omega\,,$ $f$ is non-negative and $d(x,\partial \omega)\geq R$ for $x\in B_R(y)$,
it follows that
\begin{equation*}
 \int_{B_{R}(y)} u\, dx = \int_{\omega} u\big(-\Delta \xi + \|a\|_{\infty}\xi\big)\, dx 
 \geq  \int_{\omega} f(x)\,\xi\, dx \geq c_3 \int_{\omega} f(x) \,d(x,\partial \omega)\, dx  
 \geq c_3 R \int_{B_R(y)} f(x)\, dx\,.
\end{equation*}
Hence, substituting the above information in \eqref{bc1} we obtain for $c_4 = c_2c_3R$ 
\begin{equation} \label{bc2}
u(x) \geq c_4 \int_{B_R(y)} f(x) \,dx\,, \quad \forall\ x \in  \overline{B_{R/2}(y)}\,,
\end{equation}
%\medbreak
from which, since $\omega \subset \R^N$ is bounded, \eqref{bc3} follows.
\medbreak

\noindent \textbf{Step 2:} \textit{There exists $c_5 = c_5(R,y,\omega,\|a\|_{\infty}) > 0$ such that}
\begin{equation} \label{bc5}
\frac{u(x)}{d(x,\partial \omega)} \geq c_5 \int_{B_R(y)} f(x) \,dx\,, \quad 
\forall\ x \in \omega \setminus  \overline{B_{R/2}(y)}.
\end{equation}
%\medbreak
\indent Let $w$ be the unique solution of 
\begin{equation} \label{bc4}
\left\{
\begin{aligned}
-\Delta w + \|a\|_{\infty}w & = 0 \,, \quad & \textup{ in } \omega \setminus \overline{B_{R/2}(y)}\,, \\
w & = 0\,, & \textup{ on } \partial \omega\,,\\
w & = 1\,, & \textup{ on } \partial B_{R/2}(y)\,.
\end{aligned}
\right.
\end{equation}
Still by \cite[Theorem 3]{B_N_1993}, there exists $c_6 = c_6(R,y,\omega,\|a\|_{\infty}) > 0$ such that $w(x) \geq c_6 d(x,\partial \omega)$ for all 
$x \in \omega \setminus \overline{B_{R/2}(y)}$. %{\color{red}Dependence de $c_6$}
 On the other hand, let us introduce
\[ v(x) = \frac{u(x)}{c_4 \int_{B_R(y)} f(x) \,dx}\,,\]
with $c_4$ given in \eqref{bc2}. Observe that $v$ is an upper solution of \eqref{bc4}. Hence, by the standard comparison principle, it follows that $v(x) \geq w(x)$ for all $x \in \omega \setminus  \overline{B_{R/2}(y)}$ and \eqref{bc5} follows.
\medbreak

\noindent \textbf{Step 3:} \textit{Conclusion.}
\medbreak
The result follows from \eqref{bc3} and \eqref{bc5}. 
\end{proof}

\section{Boundary weak Harnack inequality} \label{appBWHI}

In this section we present a \textit{boundary weak Harnack inequality} that will be central in the proof of Theorem \ref{aPrioriBound}. As we believe this type of inequality has its own interest, we establish it in the more general framework of the $p$-Laplacian. Recalling that $\Delta_p u = \textup{div}(|\nabla u|^{p-2}\nabla u)$ for $1<p<\infty$, we introduce the boundary value problem
\begin{equation} \label{eqBWHIP}
\pLaplac u + a(x)|u|^{p-2}u = 0\,, \quad u \in W_0^{1,p}(\omega)\,.
\end{equation} 
Let us also recall that $u \in W^{1,p}(\omega)$ is an {\it upper solution of \eqref{eqBWHIP}} if $u^{-} \in \Wop$ and, for all $\varphi \in \Wop$ with $\varphi \geq 0$, it follows that
\begin{equation*}
\int_{\omega} |\nabla u|^{p-2} \nabla u\, \nabla \varphi \, dx + \int_{\omega}  a(x)|u|^{p-2}u \,\varphi \, dx\geq  0\,. 
\end{equation*} 
We then prove the following result.

\begin{theorem} \rm \textbf{(Boundary Weak Harnack Inequality)} \label{BWHIP}
Let $\omega \subset \RN$, $N \geq 2$, be a bounded domain with boundary $\partial \omega$ of class 
$\mathcal{C}^{1,1}$ and let $a \in L^{\infty}(\omega)$ be a non-negative function. Assume that $u$ is a non-negative upper solution of \eqref{eqBWHIP}
and let $x_0 \in \partial \omega$. Then, there exist  $\overline R>0$, 
$\epsilon = \epsilon(p, \overline{R},\|a\|_{\infty},\omega) > 0$ and 
$C = C(p,\overline{R}, \epsilon,\|a\|_{\infty},\omega) > 0$ such that, for all
 $R \in(0,\overline R]\,,$
\[ \inf_{B_R(x_0) \cap \omega} \frac{u(x)}{d(x,\partial \omega)} 
\geq C \,\Big( \int_{B_R(x_0) \cap \omega} 
\Big( \frac{u(x)}{d(x,\partial \omega)} \Big)^{\epsilon} \,dx \Big)^{1/\epsilon}\,. \]
\end{theorem}
\medbreak

As already indicated, in the proof of Theorem \ref{BWHIP}  we shall make use of some ideas from \cite{S_2017}.
\medbreak

Before going further, let us introduce some notation that we will be used throughout the section. We define
\[ r := r(N,p) = \left\{
\begin{aligned}
& \frac{N(p-1)}{N-p} &\textup{ if } p < N, \\
& +\infty & \textup{ if } p \geq N,
\end{aligned}
\right.
\]

\vspace{0.05cm}

%\medbreak
\noindent and denote by $Q_{\rho}(y)$ the cube of center $y$ and side of length $\rho$, i.e.
\[ Q_{\rho}(y) = \{ x \in \R^N: |x_i-y_i| < \rho/2 \textup{ for } i = 1, \ldots, N \}.
\]

\vspace{0.05cm}
\noindent In case the center of the cube is $\rho e$ with $e = (0,0,\ldots,1/2)$, we use the notation 
$Q_{\rho} = Q_{\rho}(\rho e)$. % Finally, for the lower boundary of these cubes we will use
% $ Q_{\rho}^0 = \{x \in \partial Q_{\rho}: x_N = 0\}\,.$
\medbreak

Let us now introduce several auxiliary results that we shall need to prove Theorem \ref{BWHIP}.
%\subsection{Preliminaries}
%$ $ \medbreak
We begin recalling the following comparison principle for the $p$-Laplacian.

\begin{lemma} \rm \cite[Lemma 3.1]{T_1983} \label{CPT} \it
Let $\omega \subset \RN$, $N \geq 2$, be a bounded domain and let $a \in L^{\infty}(\omega)$ be a non-negative function. 
Assume that $u\,, v \in W^{1,p}(\omega)$ satisfy (in a weak sense)
\begin{equation*}
\left\{
\begin{aligned}
\pLaplac u + a(x)|u|^{p-2} u & \leq \pLaplac v + a(x)|v|^{p-2}v\,, \quad & \textup{ in }\omega\,, \\
u & \leq v\,, \quad & \textup{ on } \partial \omega\,.
\end{aligned}
\right.
\end{equation*}
Then, it follows that $u \leq v$.
\end{lemma}
\medbreak
As a second ingredient, we need the weak Harnack inequality.

\begin{theorem} \rm \cite[Theorem 3.13]{M_Z_1997}
\label{IWHIRefined} \it
Let $\omega \subset \RN$, $N \geq 2$, 
be a bounded domain and let $a \in L^{\infty}(\omega)$ be a non-negative function. Assume that $u \in W^{1,p}(\omega)$ 
is a non-negative upper solution of
\begin{equation*}
\pLaplac u + a(x)|u|^{p-2}u = 0\,, \quad u \in W_0^{1,p}(\omega)\,,
\end{equation*} 
and let $Q_{\rho}(x_0) \subset \omega$. Then, for any $\sigma,\tau \in (0,1)$ and $\gamma \in (0,r)$,
 there exists $C = C(p,\gamma,\sigma,\tau,\rho,\|a\|_{\infty}) > 0$ such that
\[ \inf_{Q_{\tau \rho}(x_0)} u \geq 
C \Big( \int_{Q_{\sigma \rho}(x_0)} u^{\gamma} \,dx \Big)^{1/\gamma}\,.\]
\end{theorem}

In the next result,  we deduce a more precise information on the dependence of $C$ with respect to 
$\rho$. This is closely related to  \cite[Theorem 1.2]{T_1967} where however the constant still depends on 
$\rho$.

\begin{cor} 
 \label{IWHICubes} 
Let $a$ be a non-negative constant and $\gamma \in (0,r)$. There exists $C = C(p,\gamma,a) > 0$ such that, for all  $0 < \tilde{\rho} \leq 1$, any 
 $u \in W^{1,p}(Q_{\frac{3\tilde \rho}{2}}(e))$  non-negative upper solution of 
\begin{equation} \label{equ CorA4 u}
\pLaplac u + a |u|^{p-2}u = 0\,, \quad u \in W_{0}^{1,p}(Q_{\frac{3\tilde \rho}{2}}(e)),
\end{equation} 
satisfies
%Then, there exists $C = C(N,p,\gamma,a) > 0$ such that
\[ 
\inf_{Q_{\tilde \rho}(e)} u 
\geq 
C\, {\tilde \rho}^{\,-N/\gamma} \Big( \int_{Q_{\tilde \rho}(e)} u^{\gamma} \,dx \Big)^{1/\gamma}\,.
\]
\end{cor}

\begin{proof}
Let $C = C(p,a,\gamma) > 0$ be the constant given by Theorem \ref{IWHIRefined} applied with $ \rho=\frac{3}{2}$ and $\sigma=\tau=\frac{2}{3}$. This means that if $v \in W^{1,p}(Q_{\frac{3}{2}}(e))$ is a non-negative upper solution of
\begin{equation} \label{eq cub 32}
\pLaplac v + a |v|^{p-2}v = 0\,, \quad \quad v \in  W_0^{1,p}(Q_{\frac{3}{2}}(e)),
\end{equation} 
then 
\[ \inf_{Q_{1}(e)} v(y) \geq 
C \Big( \int_{Q_{1}(e)} v^{\gamma} \,dy \Big)^{1/\gamma}\,.
\]
As $0 < \tilde{\rho} \leq 1$, observe that if $u$ is a non-negative upper solution of \eqref{equ CorA4 u}, then $v$ defined by $v(y)=u(\tilde \rho y', \tilde \rho (y_N-\frac12)+\frac12)$, where $y=(y',y_N)$ with $y'\in \mathbb R^{N-1}$, is a non-negative upper solution of \eqref{eq cub 32}.
%satisfies 
%\begin{equation*}
%\pLaplac v + a |v|^{p-2}v \geq (-{\tilde \rho}^{\,p}+1)  a |v|^{p-2}v\geq 0 \quad \textup{ and }
% \quad v\geq 0, \quad \textup{ in } Q_{\frac{3}{2}}(e).
%\end{equation*} 
Thus, we can conclude that
\begin{equation*}
\inf_{Q_{\tilde \rho}(e)} u(x) = \inf_{Q_{1}(e)} v(y) \geq 
C \Big( \int_{Q_{1}(e)} v^{\gamma} \,dy \Big)^{1/\gamma}
=
C {\tilde \rho}^{\,-N/\gamma} \Big( \int_{Q_{\tilde \rho}(e)} u^{\gamma} \,dx \Big)^{1/\gamma}\,.
\end{equation*}
\end{proof}
\medbreak
Finally, we introduce a technical result of measure theory.

\begin{lemma} \rm \cite[Lemma 2.1]{I_S_2016} \it  \label{GISL}
Let $E \subset F \subset Q_1$ be two open sets. Assume  there exists $\alpha \in (0,1)$ such that:
\begin{itemize}
\item $|E| \leq (1-\alpha)|Q_1|$.
\item For any cube $Q \subset Q_1$, $|Q \cap E| \geq (1-\alpha)|Q|$ implies $Q \subset F$.
\end{itemize}
Then, it follows that $|E| \leq (1-c\alpha)|F|$ for some constant $c = c(N) \in (0,1)$.
\end{lemma}

Now, we can perform the proof of the main result. We prove the boundary weak Harnack inequality for cubes and as consequence we obtain the desired result.

%\subsection{The proof}

\begin{lemma}[\textbf{Growth lemma}] \label{GL}
Let $a$
% \in L^{\infty}(Q_4)$ 
be a non-negative constant. Given $\nu > 0$, 
there exists $k = k(p,\nu, a) > 0$ such that, if $u \in W^{1,p}(Q_{\frac32})$ is a non-negative upper solution of 
%non-negative weak solution of 
\begin{equation*}
\pLaplac u + a\,|u|^{p-2}u = 0\,, \quad u \in W_{0}^{1,p}(Q_{\frac32} )\,,
\end{equation*}
and the following inequality holds
\begin{equation} \label{hypGL}
|\{x\in  Q_1 : u(x) > x_N \}| \geq \nu.
\end{equation}
Then $u(x) > k x_N$ in $Q_1$.
\end{lemma}

% Versión vieja Growth Lemma for cubes
\begin{remark}
Before we prove the Lemma, observe that there is no loss of generality in considering $a$ a non-negative constant instead of $a\in L^{\infty}(Q_{\frac32})$ non-negative. If $u\geq 0$ satisfies
\begin{equation*}
\pLaplac u + a(x)|u|^{p-2}u \geq 0\,, \quad \textup{ in } Q_{\frac32},
\end{equation*}
then $u$ satisfies also
\begin{equation*}
\pLaplac u + \|a\|_{\infty}\,|u|^{p-2}u \geq 0\,, \quad \textup{ in } Q_{\frac32}.
\end{equation*}
\end{remark}

\begin{proof}
%Let us fix a subdomain $\Omega$ with boundary $\partial \Omega$ of class $\mathcal{C}^{1,1}$ such that
%\[ Q_3 \subset \Omega \subset Q_4 \quad \textup{ and }
 %\quad d(\partial \Omega, \partial Q_4 \setminus Q_4^0) > 0\,.\]
%%For instance, for $N = 2$, we can think of the following picture.
%%\begin{center}
%%\begin{figure}[h]
%%\includegraphics[scale=0.6]{Fig1}
%%\caption{Illustration of the construction of the proof of Lemma \ref{GL}}
%%\end{figure}
%%\end{center}
%Moreover, let us denote $d(x) = d(x,\partial \Omega)$ and define
Let us define
$\Sd =   Q_{\frac32} \setminus Q_{\frac32-\delta}\big(\frac32 e\big)$
and  fix $c_1 = c_1(\nu) \in (0,\frac12)$ small enough in order to ensure that 
$|\Sd| \leq \frac{\nu}{2}$ for any $0 < \delta \leq c_1$. 
\medbreak

\noindent \textbf{Step 1:} \textit{For all $\delta\in (0, c_1]$, it follows that
$ |\{x \in Q_{\frac32-\delta}\big(\frac32 e\big): u(x) > x_N \} | \geq \frac{\nu}{2}$.}
\medbreak

Directly observe that
\[
\{x\in  Q_1 : u(x) > x_N \}\subset 
\{x\in  Q_{\frac32}: u(x) > x_N \}
\subset 
\{x\in  Q_{\frac32-\delta}\big(\frac{3}{2}e\big) : u(x) > x_N \}
\cup \Sd.
\]
Hence, Step 1 follows from \eqref{hypGL} and the choice of $c_1$.
%we have
%|\{x\in  Q_1 : u(x) > x_N \}| \geq \nu.
%As $\Omega=\Od \cup \Sd$, we have
%%First of all, observe that $\{ x \in \Omega: u(x) \geq d(x)\} \subset \{ x \in \Sd: u(x) \geq d(x)\} \cup \Od$. 
%%Hence, it follows that
%\[ |\{ x \in \Omega: u(x) \geq d(x) \} | \leq |\{ x \in \Sd: u(x) \geq d(x)\}|+ |\Od|\,.\]
%Then the result follows from Step 1 and the assumption  $|\Od| \leq \frac{\nu}{2}$.
%%, we deduce that
%%\[ |\{ x \in \Sd: u \geq d\}| \geq |\{ x \in \Omega: u \geq d \} | - |\Od| \geq \frac{\nu}{2} \,.\]
\medbreak

\noindent \textbf{Step 2:} \textit{For any $\epsilon > 0$ and  all $\delta\in (0, c_1]$,  
the following inequality holds}
\begin{equation} \label{s3GL}
\Big( \int_{Q_{\frac32-\delta}\big(\frac32 e\big)} u^{\epsilon} \, dx \Big)^{1/\epsilon} 
\geq \frac{\delta}{2} \big( \frac{\nu}{2} \big)^{1/\epsilon}.
\end{equation}
\medbreak

Since $u \geq 0$ and, for any $x \in Q_{\frac32-\delta}\big(\frac32 e\big)$ we have $x_N \geq \frac{\delta}{2}$, it follows that
\[
\begin{aligned}
\int_{Q_{\frac{3}{2}-\delta}(\frac32 e)} u^{\epsilon} \,dx 
&\geq
 \int_{ \{ x \in Q_{\frac32-\delta}(\frac32 e):\ u(x) \geq x_N\}} u^{\epsilon} \,dx
 \\
& \geq \int_{ \{ x \in Q_{\frac32-\delta}(\frac32 e): u(x) \geq x_N\}} 
\Big(\frac{\delta}{2}\Big)^{\epsilon} \,dx 
=
\Big(\frac{\delta}{2}\Big)^{\epsilon}\,
\Big|\Big\{ x \in Q_{\frac32-\delta}\big(\frac{3}{2}e\big): u(x) \geq x_N\Big\}\Big|.
\end{aligned}
\]
Step 2 follows then from Step 1.
%Considering together the above inequalities and applying Step 2, we conclude that
%\[ \int_{\Sd} u^{\epsilon} \,dx \geq \delta^{\epsilon} \frac{\nu}{2}\,,\]
%and so, that \eqref{s3GL} holds.
\medbreak

\noindent \textbf{Step 3:} \textit{For any $\epsilon \in (0,r)$ and  all $\delta\in (0, c_1]$,  there exists 
$C_{\delta} = C_{\delta}(p,\epsilon, \delta, a) > 0$ such that}
\[ \inf_{Q_{\frac32-\delta}\big(\frac32 e\big)} \frac{u(x)}{x_N}
 \geq \frac{ \delta \,C_{\delta}}{3} \Big( \frac{\nu}{2} \Big)^{1/\epsilon}.
 \]
\medbreak
By Theorem  \ref{IWHIRefined} applied with $\rho = \frac32$, $x_0 = \frac32 e$  and 
$\tau = \sigma = 1-\frac{2}{3}\delta$, 
there exists a constant $C_{\delta} = C_{\delta}(p,\epsilon,\delta,a) > 0$ such that
\[ \inf_{Q_{\frac32-\delta}\big(\frac32 e\big)} u (x)
\geq C_{\delta} \Big( \int_{Q_{\frac32-\delta}\big(\frac32 e\big)}  u^{\epsilon} \,dx \Big)^{1/\epsilon}\,.\]
%Hence, since $\Sd \subset Q_{4-\delta}(4e)$, it follows that
%\[ \inf_{\Sd} u \geq C_{\delta} \Big( \int_{\Sd} u^{\epsilon} \,dx \Big)^{1/\epsilon}.\]
Since for all $x\in Q_{\frac32-\delta}\big(\frac32 e\big)$ we have  $x_N\leq \frac32$, 
%\[ \frac{u(x)}{d(x,\partial \Omega)} \geq \frac{2 u(x)}{\diam(\Omega)}\,, \qquad \forall\ x \in \Omega\,.\]
%Thus, 
Step 3 follows from the above inequality and Step 2.
%we obtain that
%\[ \inf_{\Sd} \frac{u(x)}{\dist(x,\partial \Omega)} 
%\geq 
%\frac{2C}{\diam(\Omega)} \left( \int_{\Sd} u^{\epsilon} \,dx \right)^{1/\epsilon} 
%\geq 
%\frac{2C}{\diam(\Omega)} \delta \left( \frac{\nu}{2} \right)^{1/\epsilon}\,.\] 
\medbreak

\noindent \textbf{Step 4:} \textit{Conclusion.}
%There exists $k = k(N,p,\nu, a) > 0$ such that
 %$u(x) > k x_N$ in $Q_1$.} %{\color{red}$k$ depend aussi de $\epsilon$ non ?}
\medbreak

We fix $\epsilon \in (0,r)$, define 
$k_{\delta} = \frac{\delta\,C_{\delta}}{3} \big( \frac{\nu}{2} \big)^{1/\epsilon}$
%If we show the existence of  $k = k(N,p,\nu, a) > 0$ such that 
%$\displaystyle\lim_{\delta \rightarrow 0} k_{\delta} \geq 2k\,,$ applying Step 3 we can conclude. 
and introduce $\eta:[-\frac{3-2c_1}{4},\frac{3-2c_1}{4}]^{N-1} \to \R$ a $\mathcal{C}^{\infty}$ function satisfying
% $\eta\in \mathcal{C}^{\infty}([-\frac{3-2c_1}{4},\frac{3-2c_1}{4}]^{N-1},\mathbb R)$
\[ \eta(x_1, \ldots, x_{N-1}) = \left\{ \begin{aligned}
& 0, \quad &\textup{if } (x_1, \ldots, x_{N-1}) \in \big[-\tfrac{1}{2},\tfrac{1}{2}\,\big]^{N-1}, \\
& \tfrac{c_1}{2}, & \textup{if } (x_1, \ldots, x_{N-1}) \in 
\partial_{\mathbb R^{N-1}} \big(\big[-\tfrac{3-2c_1}{4},\tfrac{3-2c_1}{4}\big]^{N-1}\big),
\end{aligned}
\right.
\]
and 
\[
0\leq \eta(x_1, \ldots, x_{N-1}) \leq \frac{c_1}{2}\qquad  \textup{ for } (x_1, \ldots, x_{N-1}) \in 
\big[-\tfrac{3-2c_1}{4},\tfrac{3-2c_1}{4}\big]^{N-1}.
\]
Moreover, we consider the auxiliary function 
\[ 
\vd (x_1, \ldots, x_N)= \frac{1}{\delta} \big(x_N-\eta(x_1,\ldots,x_{N-1})\big)^2 + \big(x_N-\eta(x_1,\ldots,x_{N-1})\big)
\]
defined in
\[ 
\begin{aligned}
\omega_{\delta}&=\Big\{ 
 (x_1, \ldots, x_N) \in 
\big[-\tfrac{3-2c_1}{4},\tfrac{3-2c_1}{4}\big]^{N-1}
\times \big[0,  \tfrac{c_1}{2}\big] :
% \\
% &\hspace{70mm}\eta(x_1, \ldots, x_{N-1}) \leq \frac{\delta}{2} 
 % \textup{ and }
\eta(x_1, \ldots, x_{N-1}) \leq x_N\leq \frac{\delta}{2}\Big\}.
\end{aligned}
\]
Observe that, in $\Od$, we have  $0\leq x_N-\eta(x_1,\ldots,x_{N-1})\leq \frac{\delta}{2}$. Hence, there exists $c_2 = c_2(p,\nu,a) \in (0,c_1]$ such that, for all $0 < \delta \leq c_2$,
\[ 
\pLaplac \vd + a |v_{\delta}|^{p-2}v_{\delta} \leq 
-\frac{2}{\delta}(p-1) + 2^{p-1} \Big|\sum_{i=1}^{N-1}\frac{\partial}{\partial x_i}\Big[
\Big(\sum_{i=1}^{N-1}(\frac{\partial\eta}{\partial x_i})^2+1\Big)^{\frac{p-2}{2}}
\frac{\partial\eta}{\partial x_i}\Big] \Big|
+ \frac{3a}{4} \delta
\leq
0\,, \quad \textup{ in }\Od.
\]
%Following \cite[Section 14.6]{G_T_2001_S_Ed} (see also \cite[Page 130]{Q_S_2010} ), we obtain that
%\begin{equation*}
%\begin{aligned}
%\pLaplac \vd  
%%= \pLaplac \left(\frac{1}{\delta} d^2 + d\right) 
%%= -\div \left( \Bigl| \frac{1}{\delta} \nabla(d^2) + \nabla d \Bigr|^{p-2} 
%%\Bigl( \frac{1}{\delta} \nabla( d^2 ) + \nabla d \Bigr) \right) \\
%%& = -\div \left( \Bigl| \frac{2d \nabla d}{\delta} + \nabla d \Bigr|^{p-2} \Bigl( \frac{2d \nabla d}{\delta} 
%%+ \nabla d \Bigr)\right) \\
%& = -\div \Big( \bigl( \frac{2d}{\delta}+1 \bigr)^{p-1} |\nabla d |^{p-2} \nabla d \Big) 
%= -\div \Big( \bigl( \frac{2d}{\delta}+1 \bigr)^{p-1} \nabla d \Big) \\
%& = -(p-1) \frac{2}{\delta} \bigl( \frac{2d}{\delta}+1 \bigr)^{p-2} 
%+ \bigl( \frac{2d}{\delta}+1 \bigr)^{p-1} (- \Delta d) \\
%& = -(p-1) \frac{2}{\delta} \bigl( \frac{2d}{\delta}+1 \bigr)^{p-2} + \bigl( \frac{2d}{\delta}+1 \bigr)^{p-1} 
%\Big( \sum_{i=1}^{N-1} \dfrac{\kappa_i}{1-\kappa_i d} \Big), 
%%\\
%\end{aligned}
%\end{equation*}
%where $\kappa_i$ denote the main curvatures of $\omega$. Observe that, $ 0 \leq d/\delta < 1$ in $ \Od$.
%Hence, since the boundary of $\Omega$ is smooth, there exists $c_2 = c_2(\Omega,p) \in (0,c_1)$ such that, 
%for any $\delta \in (0,c_2)$, 
%\[ \pLaplac \vd + a(x) |v_{\delta}|^{p-2}v_{\delta} \leq 0\,, \quad \textup{ in } \Od\,.\]
On the other hand, we define $u_{\delta} = \frac{2u}{k_{\delta}}$ and immediately observe that
\[ \pLaplac u_{\delta} + a |u_{\delta}|^{p-2}u_{\delta} \geq 0\,, \quad \textup{ in } \Od\,.\]
Now, since by Step 3, we have 
\[
u_{\delta}\geq \frac{2k_{\delta}}{k_{\delta}}\frac{\delta}{2}=\delta\geq v_{\delta},
\quad \textup{ for  } x_N=\frac{\delta}{2},
\]
it follows that
\[
u_{\delta}\geq v_{\delta}
\quad \textup{ on  }\partial \Od.
\]
%\[ u_{\delta} \geq v_{\delta} = 0\, \quad \textup{ on } \partial \Omega\,,\]
%and, by Step 4, 
%\[ u_{\delta} = \frac{2u}{k_{\delta}} \geq \frac{2k_{\delta} d}{k_{\delta}} = 2d = 2\delta = v_{\delta}\,, 
%\quad \textup{ on }  \partial \Od \setminus \partial \Omega\,,\]
%we have that
%\[ u_{\delta} \geq v_{\delta} \quad \textup{on} \quad \partial \Od\,.\]
%%In short, we have
%%\begin{equation*}
%%\left \{
%%\begin{aligned}
%%\pLaplac v_{\delta} + a(x)|v_{\delta}|^{p-2}v_{\delta}&
 %%\leq \pLaplac u_{\delta} + a(x)|u_{\delta}|^{p-2}u_{\delta}\,, \quad &\textup{ in } \Od\,,\\
%%v_{\delta} &\leq u_{\delta}\,, \quad &\textup{ on } \partial \Od\,.
%%\end{aligned}
%%\right.
%%\end{equation*}
Then, applying Lemma \ref{CPT}, it follows that, for any $\delta \in (0,c_2]$,  
$v_{\delta} \leq u_{\delta}$ in $\Od$. For $\delta = c_2/2$, we obtain in particular
\begin{equation*}
 u(x)  \geq \frac{1}{2}k_{\frac{c_2}{2}} v_{\frac{c_2}{2}} (x)
 = \frac{1}{2}k_{\frac{c_2}{2}}\bigl( \frac{2}{c_2} x_N^2 + x_N \bigr) 
 \geq \frac{1}{2}k_{\frac{c_2}{2}} x_N \,, \quad \textup{ in } \omega_{\frac{c_2}{2}}\cap Q_1\,.
\end{equation*}
The result then follows from the above inequality and Step 3.
%As Step 4 implies that
%\[ u \geq k_{\frac{c_2}{2}} d\,, \quad \textup{ in } S_{\frac{c_2}{2}}\,,\]
%the result follows.
%%Hence, we can conclude that there exists $k = k(N,p,\nu,\Omega,\|a\|_{\infty})> 0$ such that 
%%$u > kd$ in $\Omega\,,$ as desired.
%\medbreak
%
%\noindent \textbf{Step 6:} \textit{Conclusion.}
%\medbreak
%
%First of all, observe that $Q_1 \subset \Omega$ and,  for all $x \in Q_1$,  $d(x) = x_N$. 
%Hence, Step 5 implies that $u(x) > kx_N$ in $Q_1$ where $k = k(N,p,\nu,\Omega, \|a\|_{\infty}) > 0$. 
%Since $\Omega$ is chosen and fixed, we can conclude that $k$ depends only in 
%$N$, $p$, $\nu$ and $\|a\|_{\infty}$.
%% {\color{green}Pourquoi ce passage par $\Omega$? Pourquoi ce jeu avec $Q_1$, $Q_3$ et $Q_4$? 
%%Que pour avoir de l'info sur $Q_1$ on a besoin de l'equation sur $Q_R$ avec $R>1$, c'est clair mais 
%%pourquoi ce rapport? Je peux l'accepter mais ce passage par $\Omega$ est lourd.}
\end{proof}

\begin{lemma} \label{iterationLemma}
Let $a \in L^{\infty}(Q_4)$ be a non-negative function. 
Assume that $u \in W^{1,p}(Q_4)$ is a non-negative upper solution of
\begin{equation} \label{IL11}
\pLaplac u + a(x)|u|^{p-2}u = 0\,, \quad u \in W_0^{1,p}(Q_{4})\,,
\end{equation}
satisfying
\[ \inf_{Q_1} \frac{u(x)}{x_N} \leq 1\,.\]
Then, there exist $M = M(p,\|a\|_{\infty}) > 1$ and $\mu  \in (0,1)$ such that
\begin{equation} \label{IL1}
\big| \{x \in Q_1: u(x)/x_N > M^j \} \big| < (1-\mu)^j\,, \quad \forall\ j \in \N\,.
\end{equation}
\end{lemma}

\begin{proof}
Let us fix some notation that we use throughout the proof. 
We fix $\gamma \in (0,r)$ and consider $C_1 = C_1(p,\|a\|_{\infty}) > 0$ the constant given by 
Corollary \ref{IWHICubes}. 
We introduce $\alpha \in (0,1)$ and fix $C_2 \in (0,1)$ the constant given by Lemma \ref{GISL}. 
Moreover, we choose $\nu = (1-\alpha)\big(\frac{1}{4}\big)^N$ and denote by 
$k = k(\nu,p,\|a\|_{\infty}) \in (0,1)$ the constant given by Lemma \ref{GL} applied to an upper solution of
\begin{equation}
\label{IL12} \pLaplac u + 2^p\|a\|_{\infty} |u|^{p-2} u = 0\,, \quad u \in W_0^{1,p}( Q_{\frac32})\,,
\end{equation}
with the chosen $\nu$. Let us point out that, if $u$ is a non-negative upper solution of \eqref{IL11},
then $u$ is a non-negative upper solution of \eqref{IL12}. Finally, we consider
\[ 
M \geq   \max \Big\{ \frac{1}{k}, \frac{4}{C_1}(1-\alpha)^{-1/\gamma}\Big\}\,,
\]
and we are going to show that \eqref{IL1} holds with $\mu = \alpha C_2$.
\medbreak
First of all, observe that $\{ x \in Q_1: u(x)/x_N > M\} \subset  \{ x \in Q_1: ku(x) > x_N\}$. 
Hence, since $\inf_{Q_1} ku(x)/x_N \leq k$,  Lemma \ref{GL} implies that
\begin{equation} \label{IL2}
|\{ x \in Q_1: u(x)/x_N > M\}|\leq |\{ x \in Q_1: ku(x)> x_N\}| 
<\nu < 1-\alpha<
1- C_2\alpha
\end{equation}
and, in particular, \eqref{IL1} holds for $j=1$. Now, let us introduce, for $j \in \N \setminus \{1\}$, 
\[ E = \{ x \in Q_1: u(x)/x_N > M^j \} \quad \textup{ and } \quad F = \{ x \in Q_1: u(x)/x_N > M^{j-1} \}\,.\]
Since $M > 1$ and $j \in \N \setminus \{1\}$, observe that \eqref{IL2} implies that
\begin{equation}
\label{eq A7.5}
 |E| = |\{ x \in Q_1: u(x)/x_N > M^j \} | \leq |\{ x \in Q_1: u(x)/x_N > M\} | \leq 1-\alpha\,,
\end{equation}
and the first assumption of Lemma \ref{GISL} is satisfied.
\bigbreak

\noindent \textbf{Claim:} \textit{For every  cube $Q_{\rho}(x_0) \subset Q_1$ such that
\begin{equation}\label{IL3}
|E \cap Q_{\rho}(x_0)| \geq (1-\alpha)|Q_{\rho}(x_0)| = (1-\alpha)\rho^N \,.
\end{equation}
we have $Q_{\rho}(x_0) \subset F$.} 
\bigbreak

%In order to prove that claim we distinguish three cases:
%\medbreak
%\noindent \textbf{Case 1:} $\rho \geq 1/16\,.$ Directly observe that \eqref{IL3} implies 
%\[ |\{x \in Q_1: u/M^j > x_N\}| \geq |E \cap Q_{\rho}(x_0)| \geq \nu\,.\]
%Hence, by Lemma \ref{GL}, $u/M^j > k x_N$ in $Q_1$ and so, by the definition of $k$, $u/x_N > M^{j-1}$ 
%in $Q_1$. Since $Q_{\rho}(x_0) \subset Q_1$, we easily deduce that $Q_{\rho}(x_0) \subset F$, as claimed.
%\medbreak
%Now, in order to treat cases 2 and 3, 
Let us denote $x_0=(x_0', x_{0_N})$ with $x_0'\in \mathbb R^{N-1}$. We define
the new variable $y = \big( \frac{x'-x_0'}{\rho'}, \frac{x_N}{\rho'} \big)$, 
where $\rho' = 2 x_{0_N}$, and the rescaled function $v(y) = \frac{1}{\rho'} u(\rho'y' + x_0', \rho'y_N)$. 
Then $v$ is a non-negative upper solution of
\begin{equation} \label{IL13}
 \pLaplac v + 2^p \|a\|_{\infty}|v|^{p-2}v = 0\,, \quad \textup{ in } Q_{4/\rho'}\big(-x_0'/\rho', 2/\rho'\,\big).
\end{equation}
Moreover, observe that
\[ x\in E \cap Q_{\rho}(x_0) \quad \textup{ if and only if } 
\quad y\in \{ y \in Q_{\rho/\rho'}(e): v(y)/M^j > y_N\}\,,\]
and so, that \eqref{IL3} is equivalent to
%{\color{green} La je ne comprends pas du tout, je suis d'accord que $x\in E \cap Q_{\rho}(x_0)$ 
%si et seulement si $y\in \{ y \in Q_{\rho/\rho'}(e): v/M^j > y_N\}$ mais a-t-on vraiment l'egalite des 
%deux ensembles? Sinon ils n'ont pas meme mesure et ... on est mal}
%and so, that \eqref{IL3} is equivalent to
\begin{equation} \label{IL4}
|\{ y \in Q_{\rho/\rho'}(e): v(y)/M^j > y_N\}| \geq (1-\alpha)|Q_{\rho/\rho'}(e)| 
= (1-\alpha)\big(\frac{\rho}{\rho'} \big)^N.
\end{equation}
Observe also that the embedding $Q_{\rho}(x_0)\subset Q_1$ implies that $\rho\leq \rho'\leq 2-\rho$ and 
$|x_{0,i}|\leq \frac{1-\rho}{2}$ for $i\in \{1,\ldots,N-1\}$. In particular, we have 
$Q_{\frac32}\subset Q_{4/\rho'}\big(-x_0'/\rho', 2/\rho'\,\big)$. Hence $v$ is an upper solution of \eqref{IL12}. 
\medbreak

Now, we distinguish two cases:
\medbreak

\noindent \textbf{Case 1:} \textit{$\rho \geq \rho'/4$}. 
%First of all, observe that $\rho'/4 \leq \rho < 1/16$ implies that $Q_4 \subset Q_{4/\rho'}(-x_0'/\rho',2/\rho')$. 
%Hence, 
Observe that $v/M^j$ is a non-negative upper solution of \eqref{IL12}. 
Moreover, as $\rho\leq\rho'$,   \eqref{IL4} implies that
\[ |\{y \in Q_1: v(y)/M^j > y_N\}| \geq |\{y \in Q_{\rho/\rho'}(e): v(y)/M^j > y_N\}| \geq \nu\,.\]
Hence, by Lemma \ref{GL}, $v(y)/M^j > k y_N$ in $Q_1$ and so, by the definition of $k$, $v(y)/y_N > M^{j-1}$ 
in  $Q_{\rho/\rho'}(e)$. This implies  that $u(x)/x_N > M^{j-1}$ in $Q_{\rho}(x_0)$.
\medbreak

\noindent \textbf{Case 2:} \textit{$ \rho < \rho'/4$}. 
%First of all, observe that  $ Q_1 \subset Q_{4/\rho'}(-x_0'/\rho', 2/\rho')$. Hence, we have that 
Recall that $v/M^j$ is a non-negative upper solution of \eqref{IL12}. Hence, $v/M^j$ is also a
non-negative upper solution of
\begin{equation*}
\pLaplac u + 2^p\|a\|_{\infty} |u|^{p-2} u = 0\,, 
\quad \textup{ in }  Q_{\frac{3 \rho}{2 \rho'}}(e) \subset Q_{\frac32}\,,
\end{equation*}
%\[ 
%\pLaplac v + 2^p \|a\|_{\infty}|v|^{p-2}v \geq 0\,, \quad \textup{ in } Q_1\,.
%\]
Thus, by Corollary \ref{IWHICubes}, we deduce that
%\begin{equation*}
%\inf_{Q_{1/4}(e)} \frac{v}{M^j} \geq C_1 \Big( \big( \frac{1}{4} \big)^{-N} 
%\int_{Q_{1/4}(e)} \big( \frac{v}{M^j} \big)^{\gamma} \,dx \Big)^{1/\gamma}\,,
%\end{equation*}
%and so, as $\rho/\rho' < 1/4\,,$ that
\begin{equation}
\label{eq LemB7 case 2} 
\inf_{Q_{\rho/\rho'}(e)} \frac{v(y)}{M^j} \geq C_1 \Big( \big( \frac{\rho}{\rho'} \big)^{-N} 
\int_{Q_{\rho/\rho'}(e)} \big( \frac{v}{M^j} \big)^{\gamma} \,dy \Big)^{1/\gamma}. 
\end{equation}
%where $C = C(N,p,\|a\|_{\infty}) > 0$ is the constant named above. 
Now, let us introduce
\[ G = \{ y \in Q_{\rho/\rho'}(e): v(y)/M^j > 1/4\}\,,\]
and, as $y_N > 1/4$ for all $y \in Q_{\rho/\rho'}(e)$, observe that \eqref{IL4} implies the following inequality
\[ |G| \geq |\{y \in Q_{\rho/\rho'}(e): v(y)/M^j > y_N \}| \geq (1-\alpha)\big( \frac{\rho}{\rho'} \big)^N\,.\]
Hence, we deduce that 
\[ 
%\begin{aligned}
\int_{Q_{\rho/\rho'}(e)} \big( \frac{v}{M^j} \big)^{\gamma} \,dy 
\geq 
\int_{G} \big( \frac{v}{M^j} \big)^{\gamma} \,dy 
>
\big(\frac{1}{4} \big)^{\gamma} |G|
 \geq 
 \big(\frac{1}{4} \big)^{\gamma} (1-\alpha) \big( \frac{\rho}{\rho'} \big)^{N},
%\end{aligned}
\]
and so, by \eqref{eq LemB7 case 2}, that
\[ \inf_{Q_{\rho/\rho'}(e)} \frac{v}{M^j} > \frac{C_1}{4} (1-\alpha)^{1/\gamma}\,.\]
Finally, using that $M \geq \frac{4}{C_1}(1-\alpha)^{-1/\gamma}$ and that $y_N \leq 1$ in $Q_{\rho/\rho'}(e)$, we deduce that $v(y) > M^{j-1}y_N$ in $Q_{\rho/\rho'}(e)$. %as for $y\in Q_{\rho/\rho'}(e)$ we have 
% $y_N\leq \frac12+ \frac{\rho}{2\rho'}\leq1$. 
Thus, we can conclude that $u(x) / x_N > M^{j-1}$ in $Q_{\rho}(x_0)$. 
\medbreak
In both cases we prove that $u(x)/x_N > M^{j-1}$ in $Q_{\rho}(x_0)$. This means that $Q_{\rho}(x_0) \subset F$ and so, the Claim is proved. 
\medbreak
Since \eqref{eq A7.5} and  the Claim hold, we can apply Lemma \ref{GISL} and we obtain that 
$|E| \leq (1-C_2\alpha)|F|\,,$ i.e.
\[ |\{ x \in Q_1: u(x)/x_N > M^j \} | \leq (1-C_2 \alpha)\, |\{x \in Q_1: u(x)/x_N > M^{j-1}\}|\,,
\quad \forall\ j \in \N \setminus \{1\}\,.\]
Iterating in $j$ and using \eqref{IL2}, the result follows with $\mu = C_2\alpha \in (0,1)$ depending only on $N$.
\end{proof}

%\begin{remark}
%If $\rho'/4 \geq 1/16$ case 2 cannot happen. In that context we just need to consider cases 1 and 3 
%changing $1/16$ by $\rho'/4$ in case 1.
%\end{remark}

\begin{theorem}[\textbf{Boundary weak Harnack inequality for cubes}]
 \label{BWHIC} 
Let $a \in L^{\infty}(Q_4)$ be a non-negative function. Assume that $u \in W^{1,p}(Q_4)$ is a 
non-negative upper solution of
\begin{equation*}
\pLaplac u + a(x)|u|^{p-2}u = 0\,, \quad u  \in  W_0^{1,p}(Q_{4}),
\end{equation*}
Then, there exist $\epsilon = \epsilon(p,\|a\|_{\infty}) > 0$ and 
$C = C(p,\epsilon,\|a\|_{\infty}) > 0$ such that
\[ \inf_{Q_1} \frac{u(x)}{x_N} \geq 
C \Big( \int_{Q_1} \big( \frac{u(x)}{x_N} \big)^{\epsilon} \,dx \Big)^{1/\epsilon}\,.\]
\end{theorem}

\begin{proof}
Let us split the proof into three steps.
\medbreak

\noindent \textbf{Step 1:} \textit{Assume that $\inf_{Q_1} \frac{u(x)}{x_N} \leq 1$. 
Then, there exist $\epsilon = \epsilon(p,\|a\|_{\infty}) > 0$ and $C = C(p,\epsilon,\|a\|_{\infty}) > 0$
such that, for all $t\geq0$,}
\begin{equation*}
|\{ x \in Q_1: u(x)/x_N > t\}| \leq C \min\{1,t^{-2\epsilon}\}\,.
\end{equation*}

Let us define the real valued function 
\[ f(t) = |\{x \in Q_1: u(x)/x_N > t\}| \,,\]
and let $M$ and $\mu$ be the constants obtained in Lemma \ref{iterationLemma}. We define  
\[ C =\max\{(1-\mu)^{-1}, M^{2\epsilon}\} > 1 \quad \textup{ and } 
\quad \epsilon =- \frac{1}{2}\frac{\ln (1-\mu)}{\ln M} > 0\,.\]
If $t \in [0,M]$, we easily get
\[
|\{ x \in Q_1: u(x)/x_N > t\}| \leq 1\leq  C M^{-2\epsilon}\leq C \min\{1,t^{-2\epsilon}\}.
\]
Hence, let us assume $t > M > 1$. Without loss of generality, we assume $t \in [M^j, M^{j+1}]$ for some 
$j \in \N$, and it follows that
\[ \frac{\ln t}{\ln M} - 1 \leq j \leq \frac{\ln t}{\ln M}\,.\]
Since $f$ is non-increasing and $1-\mu \in (0,1)$,  the above inequality and Lemma \ref{iterationLemma} imply 
\begin{equation}
\label{ThmA8-1}
 f(t) \leq f(M^j) \leq (1- \mu)^j \leq (1-\mu)^{\frac{\ln t}{\ln M}-1} \,.
\end{equation}
Finally, observe that
\begin{equation}
\label{ThmA8-2}
 \ln \Big((1-\mu)^{\frac{\ln t}{\ln M}-1} \Big) 
= \Big( \frac{\ln t}{\ln M}-1 \Big) \ln(1-\mu) = \ln t \frac{\ln(1-\mu)}{\ln M} - \ln(1-\mu)
\leq 
-2\epsilon\ln t+\ln C =  \ln (C t^{-2\epsilon})\,.
\end{equation}
%and that
%\[ \ln (C t^{-2\epsilon}) = \ln C - 2\epsilon \ln t \,.\]
%So, taking 
%\[ C = \frac{1}{1-\mu} > 1 \quad \textup{ and } \quad \epsilon 
%= \frac{1}{2}\frac{\ln \left( \frac{1}{1-\mu}\right)}{\ln M} > 0\,,\]
%we can conclude that $ f(t) \leq C t^{-2\epsilon} = C \min \{ 1, t^{-2\epsilon}\}\textup{ for } t > M\,,$
 %as desired.
The Step 1 then follows from \eqref{ThmA8-1},  \eqref{ThmA8-2} and the fact that $\min\{1,t^{-2\epsilon}\} = t^{-2\epsilon}$ for $t \geq 1$. % and $f(t)\leq (1-\mu)^j\leq 1\leq C$.

\medbreak

\noindent \textbf{Step 2:}  \textit{Assume that $\inf_{Q_1} \frac{u(x)}{x_N} \leq 1$.
 Then, there exists $C = C(p,\epsilon,\|a\|_{\infty}) > 0$} such that
\begin{equation} \label{fubini}
\int_{Q_1} \Big( \frac{u(x)}{x_N} \Big)^{\epsilon} \,dx \leq C < + \infty\,.
\end{equation}

Directly, applying \cite[Lemma 9.7]{G_T_2001_S_Ed}, we obtain that
\[ \int_{Q_1} \Big( \frac{u(x)}{x_N} \Big)^{\epsilon} \,dx 
= \epsilon \int_0^{\infty} t^{\epsilon-1} |\{x \in Q_1: u(x)/x_N > t\}| \,dt\,.\]
Hence, \eqref{fubini} follows from Step 1.
\medbreak

\noindent \textbf{Step 3:} \textit{Conclusion.} 
\medbreak

Let us introduce the function
\[ v = \frac{u}{\inf_{y\in Q_1} \frac{u(y)}{y_N}+\beta}\,,\] where $\beta > 0$ is an arbitrary 
positive constant. Obviously, $v$ satisfies the hypothesis of Step 2. Hence, applying Step 2, we obtain that
\[ \int_{Q_1} \big( \frac{u(x)}{x_N} \big)^{\epsilon} 
\Big( \frac{1}{\inf_{y\in Q_1} \frac{u(y)}{y_N}+\beta} \Big)^{\epsilon} \,dx \leq C\,,
\]
or equivalently that
\[ 
\frac{1}{C^{1/\epsilon}}  \Big( \int_{Q_1} \big( \frac{u(x)}{x_N} \big)^{\epsilon}  \,dx \Big)^{1/\epsilon} 
\leq \inf_{Q_1} \frac{u(x)}{x_N} + \beta\,.\]
Letting $\beta \rightarrow 0$ we obtain the desired result. 
\end{proof}

%\begin{cor}\label{BWHIS}
%Let $a \in L^{\infty}(B_{2\sqrt{(N-1)+4}}^{+})$ be a non-negative function. 
%Assume that $u \in W^{1,p}(B_{2\sqrt{(N-1)+4}}^{+})$ is a non-negative solution of
%\begin{equation*}
%\left\{
%\begin{aligned}
%\pLaplac u + a(x)|u|^{p-2}u & \geq 0\,, \quad & \textup{ in } B_{2\sqrt{(N-1)+4}}^{+}\,,\\
%u & \geq 0\,, & \textup{ on } B_{2\sqrt{(N-1)+4}}^{0}\,,
%\end{aligned}
%\right.
%\end{equation*}
%Then, there exist $\epsilon = \epsilon(N,p,\|a\|_{\infty}) > 0$ and $C = C(N,p,\epsilon,\|a\|_{\infty}) > 0$ 
%such that
%\[ \inf_{B_1^{+}} \frac{u}{x_N} 
%\geq C \left( \int_{B_1^{+}} \left( \frac{u}{x_N} \right)^{\epsilon} \,dx \right)^{1/\epsilon}\,.\]
%\end{cor}

%\begin{proof}
%Using that $B_{1}^{+} \subset Q_1 \subset Q_4 \subset B_{2\sqrt{(N-1)+4}}^{+}$ and 
%Theorem \ref{BWHIC} we easily deduce that
%\[ \inf_{B_{1}^{+}} \frac{u}{x_N} \geq \inf_{Q_1} \frac{u}{x_N} 
%\geq C \left( \int_{Q_1} \left( \frac{u}{x_N} \right)^{\epsilon} \,dx \right)^{1/\epsilon}
% \geq C \left( \int_{B_{1}^{+}} \left( \frac{u}{x_N} \right)^{\epsilon} dx \right)^{1/\epsilon} \,.\]
%\end{proof}

\begin{proof}[\textbf{Proof of  Theorem \ref{BWHIP}}]
Thanks to the regularity of the boundary, there exists $\overline{R}>0$ and a diffeomorphism $\varphi$ such that
$\varphi(B_{\overline{R}}(x_0)\cap \omega)\subset Q_1$ and $\varphi(B_{\overline{R}}(x_0)\cap \partial\omega)\subset \{ x \in \partial Q_1: x_N = 0 \}$. 
The result then follows from Theorem \ref{BWHIC}. 
\end{proof}

We end this section by presenting a corollary of Theorem \ref{BWHIP}. Consider the equation
\begin{equation} \label{eqBCl}
-\Delta u + a(x)u = b(x)\,, \qquad u \in \HH\,,
\end{equation}
under the assumption
\begin{equation} \label{hypBC}
\left\{
\begin{aligned}
& \omega \subset \RN, \ N \geq 2, \textup{ is a bounded domain with boundary } \partial \omega
 \textup{ of class } \mathcal{C}^{1,1}\,,\\
& a \in L^{\infty}(\omega)\,,\ b^- \in L^p(\omega) \textup{ for some } p > N \textup{ and } b^+ \in L^1( \omega)\,,\\
& a \geq 0 \textup{ a.e. in } \omega\,.
\end{aligned}
\right.
\end{equation}

\begin{cor} \label{BWHI}
\it Under the assumption \eqref{hypBC}, assume that $u \in H^1(\omega)$ is a non-negative upper solution of
\eqref{eqBCl} and let $x_0 \in \partial \omega$. Then, there exist $\overline R>0$, $\epsilon = \epsilon(\overline{R},\|a\|_{\infty},$ $\omega) > 0$, $C_1 = C_1(\overline{R},\epsilon,\|a\|_{\infty},\omega) > 0$ and $C_2 = C_2(\omega,\|a\|_{\infty}) > 0$ such that, for all $R \in(0,\overline R]$, 
\[ \inf_{B_R(x_0) \cap \omega} \frac{u(x)}{d(x,\partial \omega)} 
\geq C_1 \Big( \int_{B_R(x_0) \cap \omega} 
\Big( \frac{u(x)}{d(x,\partial\omega)} \Big)^{\epsilon} dx \Big)^{1/\epsilon}- C_2 \|b^-\|_{L^p(\omega)}\,.\]
\end{cor}

In order to prove Corollary \ref{BWHI} we need the following lemma

\begin{lemma} \label{bcLemma2}
Let $\omega \subset \RN$, $N \geq 2$, be a bounded domain with boundary $\partial \omega$ of class 
$\mathcal{C}^{1,1}$ and let $a \in L^{\infty}(\omega)$ and $g \in L^p(\omega),\, p > N$, be non-negative functions. 
Assume that $u \in H^1(\omega)$ is a lower solution of
\begin{equation*}
-\Delta u + a(x) u = g(x)\,, \quad u \in H_0^1(\omega)\,.
\end{equation*}
Then there exists $C = C(\omega,\|a\|_{\infty})> 0$ such that
\[ \sup_{\omega}\frac{u(x)}{d(x,\partial \omega)} \leq C \|g\|_{L^p(\omega)}\,.\]
\end{lemma}

\begin{proof}
First of all, observe that it is enough to prove the result for $v$ solution of 
\begin{equation*}
\left\{
\begin{aligned}
-\Delta v + a(x)v & = g(x)\,, \quad & \textup{ in } \omega\,,\\
v & = 0\,, & \textup{ on } \partial \omega\,.
\end{aligned}
\right.
\end{equation*}
as, by the standard comparison principle it follows that $u \leq v$. 
Applying \cite[Theorem 9.15 and Lemma 9.17]{G_T_2001_S_Ed} we deduce that $v \in W_0^{2,p}(\omega)$ and there exists $C_1 = C_1(\omega,\|a\|_{\infty})> 0$ such that
\[ \|v\|_{W^{2,p}(\omega)} \leq C_1 \|g\|_{L^p(\omega)}\,.\]
Moreover, as $p > N$, by Sobolev's inequality, we have $C_2=C_2(\omega,\|a\|_{\infty})$ with
\[ \|v\|_{\mathcal{C}^{1}(\overline{\omega})} \leq C_2 \|g\|_{L^p(\omega)} \,,\]
and so, we easily deduce that
\[ v(x) \leq C_3 \|g\|_{L^p(\omega)}d(x,\partial \omega)\,, \quad \forall\ x \in \omega\,.\]
Hence, since $u \leq v$, the result follows from the above inequality.
\end{proof}

\begin{proof}[\textbf{Proof of Corollary \ref{BWHI}}]

Let  $w \geq 0$ be the solution of
\begin{equation}
\left\{
\begin{aligned}
-\Delta w + a(x)w & = b^-(x)\,, \quad & \textup{ in } \omega\,, \\
w & = 0\,, & \textup{ on } \partial \omega\,. 
\end{aligned}
\right.
\end{equation}
Observe that $v=u+w$ satisfies
\begin{equation}
\left\{
\begin{aligned}
-\Delta v + a(x)v &  \geq 0, \quad & \textup{ in } \omega\,, 
\\
v & \geq 0\,, & \textup{ on } \partial \omega\,. 
\end{aligned}
\right.
\end{equation}
Hence, by Theorem \ref{BWHIP}, there exist $\overline R>0$, 
$\epsilon = \epsilon(p, \overline{R},\|a\|_{\infty},\omega) > 0$ and 
$C = C(p,\overline{R}, \epsilon,\|a\|_{\infty},\omega) > 0$ such that, for all
 $R \in(0,\overline R]\,,$
\begin{equation}\label{last}
 \inf_{B_R(x_0) \cap \omega} \frac{v(x)}{d(x,\partial \omega)} 
\geq C \,\Big( \int_{B_R(x_0) \cap \omega} 
\Big( \frac{v(x)}{d(x,\partial \omega)} \Big)^{\epsilon} \,dx \Big)^{1/\epsilon}\,.
\end{equation}
On the other hand, by Lemma \ref{bcLemma2}, there exists $C_2 = C_2(\omega,\|a\|_{\infty}) > 0$ such that
\begin{equation}\label{bc7}
\sup_{\omega} \frac{w(x)}{d(x,\partial \omega)} \leq C_2 \|b^-\|_{L^p(\omega)}\,. 
\end{equation} 
From \eqref{last}, \eqref{bc7} and using that $u = v-w$, the corollary follows observing that $w\geq 0$ and hence $v\geq u$.
\end{proof}

\section{A priori bound} \label{III}

This section is devoted to the proof of  Theorem \ref{aPrioriBound}.
As a first step we observe that, to obtain our a priori upper  bound on the solutions of \eqref{Plambda}, we only need to control the solutions on 
$\Omega^+$. This can be proved under a weaker assumption than \eqref{A1}. More precisely, we assume
%Actually this result is a particular case of a more general property that we now present. \medskip
%
%Consider  the  boundary value problem
%\[ \label{Q} \tag{$Q$}
%-\Delta u = c(x) u + \mu(x) |\gradu|^2 + h(x)\,, \quad u \in \HLinfty\,,\]
%under the assumption
%\begin{enumerate}
%\item \label{1} $ c\,, h \in L^q(\Omega)$ for some $q > N/2$ and $\mu %\in L^{\infty}(\Omega)\,.$
%\item \label{2} $c^{+} \not\equiv 0\,.$
%\item \label{3} There exists a subdomain $\Omega_1 \subset \Omega$ with smooth boundary 
%$\partial \Omega_1$ of class $\mathcal{C}^{1,1}$ such that:
%\begin{itemize}
%\item $\supp(c^{+}) \subset \Omega_1\,.$
%\item $ c \geq 0\,,\ h \geq 0$ in $\Omega_1\,.$
%\item $0 < \mu_1 \leq \mu(x)$ in $\Omega_1\,.$
%\end{itemize}
%\end{enumerate}
\[ \label{01} \tag{$B$} 
\left\{
\begin{aligned}
&\Omega \subset \RN,\, N \geq 2, \textup{ is a bounded domain with boundary }\partial \Omega \textup{  of class }\mathcal{C}^{0,1}, 
\\
& c_+, \ c_-  \textup{ and } h \textup{ belong to } L^q(\Omega) \textup{ for some } q > N/2 \,,
\  \mu  \textup{ belong to } L^{\infty}(\Omega) \,, 
\\
& c_+(x) \geq 0, \ c_-(x) \geq 0 \textup{ and } c_-(x) c_+(x) =0 \textup{ a.e. in }\Omega,
\\
& |\Omega_{+}|> 0, \textup{ where } \Omega_{+} :=  \supp(c_{+}),
%\\
%& \textup{ there exists an open subset } \Omega_{1} \textup{ of } \Omega  
%%\textup{ with boundary } \partial \Omega_1 \textup{  of class }\mathcal{C}^{0,1},
% \textup{ such that } c \leq 0 \textup{ on } \Omega \backslash \overline\Omega_1.
%  \\
\end{aligned}
\right.
\] 
and we prove the next result.

\begin{lemma}
\label{Step 1}
 
Assume that \eqref{01} holds. Then, there exists  $M >0$ such that, for any $\lambda\in\mathbb R$,  any solution $u$  of \eqref{Plambda} satisfies
$$
-\sup_{\Omega_+} u^-  - M \, \leq \, u \,  \leq \sup_{\Omega_+}u^+  + M.
$$
\end{lemma}
\begin{remark}
Let us point out that if $c_{+} \equiv 0$, i.e. $|\Omega_{+}| = 0$, the problem \eqref{Plambda} reduces to
\begin{equation} \label{R41}
-\Delta u = -c_{-}(x) u + \mu(x) |\gradu|^2 + h(x), \quad u \in \HLinfty,
\end{equation}
which is independent of $\lambda$. If \eqref{R41} has a solution, by \cite[Proposition 4.1]{A_DC_J_T_2015} it is unique and so, we have an a priori bound.
\end{remark}
\begin{proof}
In case problem \eqref{Plambda} has no solution for any $\lambda\in \mathbb R$, there is nothing to prove. Hence, we assume  the  existence of  $\tilde\lambda\in\mathbb R$ such that $(P_{\tilde\lambda})$ has a solution  $\tilde{u}$.  We shall prove the result with  $M := 2 \|\tilde{u}\|_{\infty}$.
Let $u$ be an arbitrary solution of \eqref{Plambda}. 
\bigbreak

\noindent \textbf{Step 1:} \textit{$u  \leq \sup_{\Omega_+}u^++ M$.}
\medbreak

Setting $D:= \Omega \backslash \overline\Omega_+$ we define  
$v = \displaystyle  u - \sup_{\partial D} u^+$. We then obtain
\[ 
-\Delta v = -c_-(x)v + \mu(x)|\gradv|^2 + h(x) - c_-(x)  \sup_{\partial D}u^+
\leq -c_-(x)v + \mu(x)|\gradv|^2 + h(x)\,, \quad \textup{ in } D\,.
\]
As $v \leq 0$ on $\partial D$, the function
 $v$ is a lower solution of
\begin{equation}\label{pivot}
- \Delta z = -c_-(x)z + \mu(x)|\nabla z|^2 + h(x)\,, \qquad u \in H^1_0(D)\cap L^{\infty}(D). %Linfty,%\quad \textup{ in } D.
\end{equation}
%Now let $\tilde{u}$ be an arbitrary but fixed solution to $(Q)$. Clearly we can assume that such a solution do exists since otherwise there is nothing to prove. 
Setting
$ \tilde{v} = \tilde{u} + \|\tilde{u}\|_{\infty}$ we observe that
\[
 -\Delta \tilde{v} = -c_-(x) \tilde{v} + \mu(x)|\nabla \tilde{v}|^2 + h(x) + c_-(x) \|\tilde{u}\|_{\infty}
\geq  - c_-(x)\tilde{v}  + \mu(x)|\nabla \tilde{v}|^2 + h(x)\,, \quad \textup{ in } D\,,
\]
and thus, as $\tilde{v} \geq 0$ on $\partial D$, the function  $\tilde{v}$ is an upper solution of \eqref{pivot}.
By  \cite[Lemma 2.1]{A_DC_J_T_2014}, we know that $u$, $\tilde{u} \in H^1(\Omega) \cap W_{loc}^{1,N}(\Omega) \cap \mathcal{C}(\overline{\Omega})$ and hence,  $v$, $\tilde{v} \in H^1(D) \cap W_{loc}^{1,N}(D) \cap \mathcal{C}(\overline{D})$.
Applying \cite[Lemma 2.2]{A_DC_J_T_2014} we conclude that $v \leq \tilde{v}$ in $D$ namely, that
\[ 
  u - \sup_{\partial D} u^+ \leq \tilde{u} + \| \tilde{u}\|_{\infty} \,, \quad \textup{ in } D.
  \]
This gives that 
\[ 
  u \leq   \tilde{u} + \| \tilde{u}\|_{\infty} + \sup_{\partial D} u^+ \,, \quad \textup{ in } D,
  \]
and hence %finally, setting $M := 2 \|\tilde{u}\|_{\infty}$, we obtain that
\[ 
  u \leq   M +\sup_{\Omega_+} u^+  \,, \quad \textup{ in }  \Omega.
  \]
%\medbreak

\noindent \textbf{Step 2:} \textit{$u \geq -\sup_{\Omega_+} u^- - M$.}
\medbreak

%  
%We shall now prove that
%$$ 
%u \geq -\sup_{\Omega_1} u^- - M.
%$$ 
We now define $v =  \displaystyle u + \sup_{\partial D} u^-$ and obtain $v \geq 0$ on $\partial D$ as well as
\[
 -\Delta v = - c_-(x)v + \mu(x)|\gradv|^2 + h(x) + c_-(x) \sup_{\partial D} u^-
\geq - c_-(x)v + \mu(x)|\gradv|^2 + h(x)\,, \quad \textup{ in } D\,.
\]
Thus $v$ is an upper solution of \eqref{pivot}. Now defining $ \tilde{v} = \tilde{u} - \|\tilde{u}\|_{\infty}$, again, we have $\tilde{v} \leq 0$ on $\partial D$ as well as
\[ 
-\Delta \tilde{v} = -c_-(x) \tilde{v} + \mu(x)|\nabla \tilde{v}|^2 + h(x) - c_-(x) \|\tilde{u}\|_{\infty}
\leq - c_-(x)\tilde{v}  + \mu(x)|\nabla \tilde{v}|^2 + h(x)\,, \quad \textup{ in } D\,.
\]
Thus $\tilde{v}$ is a lower solution of \eqref{pivot}. As previously we have that 
$v$,  $\tilde{v} \in H^1(D) \cap W_{loc}^{1,N}(D) \cap \mathcal{C}(\overline{D})$
 and applying \cite[Lemma 2.2]{A_DC_J_T_2014} we obtain that $ \tilde{v} \leq v$ in $D$. Namely
\[ 
  \tilde{u} -  \|\tilde{u}\|_{\infty}  \leq u + \sup_{\partial D} u^-\,, \quad \textup{ in } D.\]
Thus 
\[ 
 u \geq  \tilde{u} -  \|\tilde{u}\|_{\infty}   -\sup_{\partial D} u^-\,, \quad \textup{ in } D,\]
 and without restriction we get that
 \[ 
  u \geq -\sup_{\Omega_+} u^- -M   \,, \quad \textup{ in }  \Omega,
  \]
 ending the proof.
 \end{proof}
 
Now, let $u \in \HLinfty$ be a solution of \eqref{Plambda}. 
Following \cite[Proposition 6.1]{A_DC_J_T_2015},  we introduce
\begin{equation}
\label{def w}
 w_i(x) = \frac{1}{\mu_i} \big(e^{\mu_i u(x)} - 1 \big) \quad \textup{ and } \quad g_i(s) 
= \frac{1}{\mu_i} \ln(1+\mu_i s), \qquad i = 1,2\,,
\end{equation} 
where $\mu_1$ is given in \eqref{A1} and $\mu_2 = \esssup \mu(x)$. Observe that 
\[ u = g_i(w_i) \quad \textup{ and } \quad 1+ \mu_i w_i = e^{\mu_i u} ,\qquad i = 1,2\,,\]
and that, by standard computations,
\begin{equation} \label{idwi}
-\Delta w_i = (1+\mu_i w_i)\big[(\lambda c_{+}(x)-c_{-}(x))g_i(w_i) + h(x)\big] 
+ e^{\mu_i u} |\gradu|^2 (\mu(x)-\mu_i).
\end{equation}
%For future reference note also that $\min_{[-1/\mu_i,+\infty[} (1+\mu_i s) g_i(s)=-\frac{e^{-1}}{\mu_i}$ .

Using \eqref{idwi} we shall obtain a uniform a priori upper bound on $u$ in a neighborhood of any fixed point 
$\overline{x} \in \overline{\Omega}_{+}$. We consider the two cases $\overline{x} \in \overline{\Omega}_{+} \cap \Omega$ and $\overline{x} \in \overline{\Omega}_{+} \cap \partial \Omega$ separately.

\begin{lemma} \label{Steps A}
Assume that \eqref{A1} holds and that $\overline{x} \in \overline{\Omega}_{+} \cap \Omega$. 
For each $\Lambda_2 > \Lambda_1 > 0$, there exist $M_I > 0$ and $R > 0$ such that, for any 
$\lambda \in [\Lambda_1,\Lambda_2]$, any solution $u$ of \eqref{Plambda} satisfies 
$\sup_{B_R(\overline{x})}u  \leq M_I$.

\end{lemma}

\begin{proof}
Under the assumption \eqref{A1} we can find a $R > 0$ such that 
$ \mu(x) \geq \mu_1 > 0$, $ c_{-} \equiv 0$ in $B_{4R}(\overline{x}) \subset \Omega$
and $c_{+} \gneqq 0$ in $B_R(\overline{x})$. 
For simplicity, in this proof, we denote $B_{mR}=B_{mR}(\overline{x})$, for $m\in \mathbb N$.
\medbreak
Since $c_{-} \equiv 0$ and $\mu(x)\geq\mu_1$ in $B_{4R}$, 
observe that \eqref{idwi}
reduces to
\begin{equation}
\label{equl*} 
 -\Delta w_1 + \mu_1 h^{-}(x) w_1  \geq \lambda (1+ \mu_1 w_1) c_{+}(x)g_1(w_1) +h^+(x)(1+\mu_1 w_1)
 - h^{-}(x)\,, \quad  \textup{ in }  B_{4R}.
\end{equation}
\noindent Let $z_2$ be the solution of
\begin{equation}
\label{z2} 
%\left\{
%\begin{aligned}
 -\Delta z_2 + \mu_1 h^{-}(x) z_2  = -\Lambda_2 c_{+}(x)\frac{e^{-1}}{\mu_1} \,, \qquad z_2  \in H^1_0(B_{4R}).
 %z_2  & = 0\,, \quad & \textup{ on } \partial B_{2R}\,.
%\end{aligned}
%\right.
\end{equation}
By classical regularity arguments (see for instance \cite[Theorem III-14.1]{L_U_1968}),
$z_2 \in \mathcal{C}(\overline{B_{4R}})$. Hence, there exists 
$D=D(\overline{x},\mu_1, \Lambda_2, \|h^-\|_{L^q(B_{4R})}, \|c_+\|_{L^q(B_{4R})}, q, R) > 0$ 
such that 
\begin{equation}
\label{z2borne}
z_2\geq -D \textup{ in }B_{4R}.
\end{equation} 
Moreover, by the weak maximum principle \cite[Theorem 8.1]{G_T_2001_S_Ed}, we have that $z_2 \leq 0$. Now defining $v_1=w_1-z_2+\frac{1}{\mu_1}$, and since  $\min_{[-1/\mu_i,+\infty[} (1+\mu_i s) g_i(s)=-\frac{e^{-1}}{\mu_i}$, we observe that $v_1$ satisfies
\begin{equation}
\label{v_11} 
 -\Delta v_1 + \mu_1 h^{-}(x) v_1  \geq \Lambda_1 c_{+}(x) (1+ \mu_1 w_1) g_1(w_1)^+\,, 
 \quad  \textup{ in } B_{4R}.
\end{equation}
Also, since $w_1>-1/\mu_1$, we have $v_1>0$ in  $\overline{B_{4R}}$.
%since $v_1 >0$ on $\partial B_{4R}$, by the weak maximum principle we deduce that $v_1 \geq 0$ on $\overline{B_{4R}}.$ 
Note also that $ 0 < 1+ \mu_1 w_1 = \mu_1 v_1 + \mu_1 z_2$ in  $\overline{B_{4R}}$. Now, we split the rest of the proof into four steps.
\medbreak

\noindent \textbf{Step 1:} \textit{There exist 
$C_{1} = C_{1}(\overline{x}, \Lambda_1, \Lambda_2, R, \mu_1, q, \|h^{-}\|_{L^{\infty}(B_{4R})}, 
%\|c_+\|_{L^1(B_R)}, 
\|c_+\|_{L^q(B_{4R})}) > 0 $ such that}
\begin{equation}\label{eq-step1}
  k:= \inf_{B_R}  v_1(x) \leq C_1.
\end{equation}

In case $  \mu_1 \inf_{B_R}  v_1(x)  \leq 1+\mu_1 D$, where $D$ is given by \eqref{z2borne}, the Step 1 is proved. Hence, we assume that
\begin{equation}
\label{cas2}
\mu_1 v_1(x)  \geq 1+\mu_1 D, \qquad \forall\ x\in B_R.
\end{equation}
%Observe that
%$$ k =\inf_{B_R}  v_1(x) \geq \inf_{B_{2R}}  v_1(x).
%$$
In particular, $\mu_1 v_1 + \mu_1 z_2\geq 1$ on $B_R$. Now, by Lemma \ref{bcLemma1} applied on \eqref{v_11} with $\omega = B_{4R}$, there exists $C = C(R, 
%q, 
\|h^{-}\|_{L^{\infty}(B_{4R})},\mu_1, \Lambda_1, \overline x) > 0$ such that,
\begin{equation*}
\begin{aligned}
%\mu_1 \inf_{B_R} \frac{v_1(y)}{d(y,\partial B_{2R})}
k
&\geq C \int_{B_R}  c_{+}(y)\, \Big(\mu_1 v_1(y) + \mu_1 z_2(y)\Big) \ln \Big(\mu_1 v_1(y) + \mu_1 z_2(y)\Big)\,dy 
\\
&\geq C \int_{B_R}  c_{+}(y)\,   (\mu_1 k-\mu_1 D) 
\ln\big(\mu_1 k-\mu_1 D \big) \,dy \\
&  =C (\mu_1 k-\mu_1 D) \ln\big(\mu_1 k-\mu_1 D \big)\|c_{+}\|_{L^1(B_R)}.
\end{aligned}
\end{equation*}
 As $c_{+} \gneqq 0$ in $B_R$,
 % and in view of \eqref{cas2}, 
 comparing the growth in $k$ of the various terms, we deduce that $k$ must remain bounded and thus the existence of 
 $C_{1} = (\overline{x}, \Lambda_1, \Lambda_2, R, \mu_1, q, \|h^{-}\|_{L^{\infty}(B_{4R})}, 
 %\|c_+\|_{L^1(B_R)}, 
 \|c_+\|_{L^q(B_{4R})}) > 0 $
such that \eqref{eq-step1} holds.
\medbreak

\noindent \textbf{Step 2:}\,\,\textit{For any $1 \leq s < \frac{N}{N-2}$, there exists 
 $C_{2} = C_{2}(\overline{x}, \mu_1, R, s, \Lambda_1, \Lambda_2, q, \|h^{-}\|_{L^{\infty}(B_{4R})},\, \|c_+\|_{L^q(B_{4R})}
 % \|c_+\|_{L^1(B_R)}
 ) > 0 $ such that 
%{\color{red}Verifier la dependence de la constante} 
\[ 
\int_{B_{2R}} (1+ \mu_1 w_1)^s \,dx\leq C_{2}.
\]}

Applying Lemma \ref{WHI} to \eqref{v_11}, we deduce the existence of 
$C = C(s,\mu_1, R,\|h^{-}\|_{L^q(B_{4R})}) > 0$ such that 
\[ 
\Big( \int_{B_{2R}} v_1^s \,dx \Big)^{1/s} \leq C \inf_{B_R} v_1 \,.
\]
The Step 2 follows from Step 1 observing that 
$ 0 \leq 1+ \mu_1 w_1 = \mu_1 v_1 + \mu_1 z_2 \leq \mu_1 v_1.$ 

%Hence, we can conclude that there exists $C_{2} = C_{2}(C_{1},\Omega,s) > 0$ such that
%\[ \left( \int_{B_{2R}} w_1^s dx \right)^{1/s} \leq C_{2}\,.\]
\medbreak

\noindent \textbf{Step 3:}\,\,\textit{For any $1 \leq s < \frac{N}{N-2}$, we have, for the constant $C_2 >0$ introduced in Step 2, that}
\[ \int_{B_{2R}} \big(1+\mu_2 w_2 \big)^{\frac{\mu_1 s}{\mu_2}} dx \leq 
C_{2}.\]
\indent This directly follows from Step 2 since, by the definition of $w_i$, we have 
\[
(1+\mu_2 w_2)^{\frac{\mu_1}{\mu_2}}=(e^{\mu_2 u})^{\frac{\mu_1}{\mu_2}} 
= e^{\mu_1 u}= (1+\mu_1 w_1).
\]

%Now, observe that Step 2 shows that there exists $\overline{C}_{2}> 0$ such that
%\[ \left( \int_{B_{2R}}(1+\mu_1 w_1)^{s} dx \right)^{\frac{\tau-1}{\tau}} \leq \overline{C}_{2}\,.\]
%Hence, it follows that
%\[ \int_{B_{2R}} \left( c_{+}(x) (1+\mu_2 w_2 )^{\alpha} \right)^r dx 
%\leq \overline{C}_{2} \|c_{+}\|_{L^q(B_{2R})}^r\,.\]
%\medbreak

\noindent \textbf{Step 4:}\textit{ Conclusion.} 
\medbreak

We will show the existence of  $C_{3} = C_3 (\overline{x},\mu_1, \mu_2, R,  \Lambda_1, \Lambda_2,q,\|h^{-}\|_{L^{\infty}(B_{4R})}, \|c_+\|_{L^q(B_{4R})}
%, \|c_+\|_{L^1(B_R)}
) > 0 $ such that 
%{\color{red} Verifier la dependence de la constante}
\[ \sup_{B_R} w_2 \leq C_{3}\,.\]
Thus, thanks to the definition of $w_2$, we can conclude the proof. Let us fix $s\in [1, \frac{N}{N-2})$,  
$r\in(\frac{N}{2}, q)$ 
and $\alpha = \frac{(q-r)\mu_1 s}{\mu_2 q r}$ and let $c_{\alpha}>0$ such that  
\[ \ln(1+x) \leq (1+x)^{\alpha} + c_{\alpha}, \quad \forall\ x \geq 0.\]
We introduce the auxiliary functions
\[ 
\begin{array}{c}
a(x) = \Lambda_2 c_{+}(x)(1+\mu_2w_2)^{\alpha} + c_{\alpha} \Lambda_2  c_{+}(x) + \mu_2 h^+(x), \vspace{0.225cm} \\
%\quad  \textup{and} \quad 
\displaystyle
b(x) = \frac{\Lambda_2}{\mu_2}c_{+}(x) (1+\mu_2 w_2)^{\alpha} 
+ c_{\alpha} \frac{\Lambda_2}{\mu_2} c_{+}(x) + h^+(x)+ c_-(x) \frac{e^{-1}}{\mu_2},
\end{array}
\]
and, as $\mu(x)\leq \mu_2$, we deduce from \eqref{idwi} that $w_2$ satisfies
\begin{equation*}
\left\{
\begin{aligned}
-\Delta w_2 & \leq a(x) w_2 + b(x)\, \quad & \textup{ in } \Omega\,, \\
w_2 & = 0 & \textup{ on } \partial\Omega.
\end{aligned}
\right.
\end{equation*}
Now, as $q/r > 1$, by Step 3 and H\"older inequality, it follows that 
\[ 
\begin{aligned}
\int_{B_{2R}} ( c_{+}(x) (1+\mu_2 w_2 )^{\alpha})^r dx 
&\leq 
\|c_{+}\|_{L^q(B_{2R})}^r  
\Big( \int_{B_{2R}}(1+\mu_2 w_2)^{\frac{\alpha q r}{q-r}} dx \Big)^{\frac{q-r}{q}}
\\
& \leq \|c_{+}\|_{L^q(B_{2R})}^r 
\Big( \int_{B_{2R}}(1+\mu_2 w_2)^{\frac{\mu_1 s}{\mu_2}} dx \Big)^{\frac{q-r}{q}} 
\leq C_2^{\frac{q-r}{q}} \|c_{+}\|_{L^q(B_{2R})}^r.
\end{aligned}
\]
Hence, there exists $D (\overline{x},\mu_1, \mu_2, s, R, \Lambda_1, \Lambda_2,q,\|h^{-}\|_{L^{\infty}(B_{4R})}, \|c_+\|_{L^q(B_{4R})}, 
%\|c_+\|_{L^1(B_R)}
r, \|h^+\|_{L^q(B_{2R})}) > 0 $  such that 
%{\color{red}Verifier la dependence de la constante}
\begin{equation} \label{ab6}
\max \{ \, \|a\|_{L^r(B_{2R})}, \|b\|_{L^r(B_{2R})} \} \leq D\,.
\end{equation}
Applying then Lemma \ref{LMP}, there exists $C (\overline{x},\mu_1, \mu_2, s, R, \Lambda_1, \Lambda_2,q,\|h^{-}\|_{L^q(B_{4R})},
 \|c_+\|_{L^q(B_{4R})}
 %, \|c_+\|_{L^1(B_R)}
 ) > 0 $ 
%{\color{red}Verifier la dependence de la constante} 
such that
\[ \sup_{B_R} w_2^+ \leq 
C \Big[ \Big( \int_{B_{2R}} (w_2^+)^{ \frac{\mu_1}{\mu_2} s } dx \Big)^{ \frac{\mu_2}{\mu_1 s}} 
+ \|b\|_{L^r(B_{2R})} \Big] \leq C \Big[ \Big( \int_{B_{2R}} (w_2^+)^{ \frac{\mu_1}{\mu_2} s } dx \Big)^{ \frac{\mu_2}{\mu_1 s}} 
+ D \Big] \,.\]
On the other hand, by Step 3, we get
\[  \int_{B_{2R}} (w_2^+)^{ \frac{\mu_1}{\mu_2} s } dx 
\leq C(\mu_1, \mu_2, s) \int_{B_{2R}} ( 1+\mu_2 w_2 )^{ \frac{\mu_1}{\mu_2} s } dx \leq  C(\mu_1, \mu_2, s)\,C_2\,,\]
%On the other hand, using \eqref{ab6}, we deduce that $C$ depends actually of 
%$C_{2},\Lambda_2, \mu_2, \|c^{+}\|_q$ and $ \|h\|_q $. Gathering together all this information, 
%we conclude that there exists $C_{3} = C_{3} (C_{2}, \mu_2, \Lambda_2, \|c^{+}\|_q,\|h\|_q) > 0$ such that
%\[ \sup_{B_R} w_2 \leq C_{3}\,.\] 
and the result follows.
\end{proof}

\begin{lemma} \label{Steps B}
Assume that \eqref{A1} holds and that $\overline{x} \in \overline{\Omega}_{+} \cap \partial \Omega$. For each $\Lambda_2 > \Lambda_1 > 0$, there exist $R > 0$ and $M_B > 0$  such that, 
for any $\lambda \in [\Lambda_1,\Lambda_2]$, any solution of \eqref{Plambda} satisfies 
$\sup_{B_R(\overline{x}) \cap \Omega}u \leq M_B\,.$
\end{lemma}

\begin{proof} Let $\overline R>0$ given by Theorem \ref{BWHIP}.
Under the assumption \eqref{A1}, we can find $R \in (0, \overline R/2]$ and $\Omega_1 \subset \Omega$ with 
$\partial\Omega_1$ of class ${\mathcal C}^{1,1}$ such that 
$B_{2R}(\overline{x}) \cap \Omega\subset \Omega_1$ and 
$\mu(x) \geq \mu_1 > 0$, $c_{-} \equiv 0$ and $c_{+} \gneqq 0$ in $\Omega_1$. 
\medbreak
Since $c_{-} \equiv 0$ and $\mu(x)\geq\mu_1$ in $\Omega_1$, observe that \eqref{idwi}
reduces to
\begin{equation}
\label{equ*bis}  
 -\Delta w_1 + \mu_1 h^{-}(x) w_1  \geq \lambda (1+ \mu_1 w_1) c_{+}(x)g_1(w_1) +h^+(x)(1+\mu_1 w_1) - h^{-}(x)\,, \quad  \textup{ in } \Omega_1
\end{equation}
\noindent Let $z_2$ be the solution of
\begin{equation}
\label{z2bis} 
 -\Delta z_2 + \mu_1 h^{-}(x) z_2  = -\Lambda_2 c_{+}(x)\frac{e^{-1}}{\mu_1}   
 \,, \qquad z_2 \in H^1_0(\Omega_1).
\end{equation}
As in Lemma \ref{Steps A},  $z_2\in \mathcal{C}(\overline{\Omega_1})$ and there exists a 
$D=D(\mu_1, \Lambda_2, \|h^-\|_{L^q(B_{4R})}, \|c_+\|_{L^q(B_{4R})}, q, \Omega_1) > 0$
 such that $-D \leq z_2 \leq 0$ on $\Omega_1$. Now defining  $v_1=w_1-z_2+\frac{1}{\mu_1}$ we observe that $v_1$ satisfies
\begin{equation}
\label{v_1} 
 -\Delta v_1 + \mu_1 h^{-}(x) v_1  \geq \Lambda_1 c_{+}(x)(1+ \mu_1 w_1) g_1(w_1)^+\,, 
 \quad  \textup{ in } \Omega_1.
\end{equation}
and $v_1>0$ on $\overline{\Omega_1}$.
%Since $v_1 >0$ on $\partial \Omega_1$, by the weak maximum principle  \cite[Theorem 8.1]{G_T_2001_S_Ed} %we deduce that $v_1 \geq 0$ on $\overline{\Omega_1}.$ 
Note also that $0 < 1+ \mu_1 w_1 = \mu_1 v_1 + \mu_1 z_2$ on $\overline{\Omega_1}$. Next, we split the rest of the proof into three steps.
\medbreak

\noindent \textbf{Step 1: }\textit{There exists
$C_{1} = C_{1}(\Omega_1, \overline{x}, \Lambda_1, \Lambda_2, R, \mu_1, q , \|h^{-}\|_{L^{\infty}(\Omega_1)},
%, \|c_+\|_{L^1(\Omega_1)},
 \|c_+\|_{L^q(\Omega_1)}) > 0 $
  such that 
\[ \inf_{B_{2R}(\overline{x}) \cap \Omega_1} \frac{v_1(x)}{d(x,\partial \Omega_1)} \leq C_{1}\,.\]}

\smallbreak
Choose $R_2  > 0$ and $y\in \Omega$ such that $B_{4 R_2}(y) \subset  B_{2R}(\overline{x}) \cap \Omega$ and 
$c_{+} \gneqq 0$ in $B_{R_2}(y)$.
As in Step 1 of Lemma \ref{Steps A}, 
there exists $C = C(\Omega_1, y, \Lambda_1, \Lambda_2, R_2, \mu_1, q , \|h^{-}\|_{L^{\infty}(\Omega_1)}, 
%\|c_+\|_{L^1(B_{R_2}(y))}, 
\|c_+\|_{L^q(\Omega_1)}) > 0 $
 such that 
\[ 
\inf_{B_{R_2}(y)} v_1(x) \leq C\,.\]
We conclude by observing, since $B_{4R_2}(y) \subset B_{2R}(\overline{x}) \cap \Omega_1$, that
$$ 
\inf_{B_{2R}(\overline{x}) \cap \Omega_1} \frac{v_1(x)}{d(x,\partial \Omega_1)}
\leq  \inf_{B_{R_2}(y)} \frac{v_1(x)}{d(x,\partial \Omega_1)} 
\leq 
\frac{1}{3 R_2} \, \inf_{B_{R_2}(y)} v_1(x).
$$

\noindent \textbf{Step 2:} \textit{There exist  $\epsilon = \epsilon(\overline{R}, \mu_1, \|h^{-}\|_{L^{\infty}(\Omega_1)},  \Omega_1)> 0$  and $C_{2} = C_{2}(\overline{x},\mu_1, R, \overline{R}, s, \Lambda_1, \Lambda_2,q,\|h^{-}\|_{L^{\infty}(\Omega_1)},$  $\|c_+\|_{L^q(\Omega_1)}) > 0 $  such that
\[ 
\Big( \int_{B_{2R}(\overline{x}) \cap \Omega} (1+ \mu_1 w_1)^{\epsilon} \,dx \Big)^{1/\epsilon} 
\leq
C_{2}.
\]}

\indent  By Theorem \ref{BWHIP} applied on \eqref{v_1}  and Step 1, we obtain constants 
$\epsilon = \epsilon(\overline{R}, \mu_1, \|h^{-}\|_{L^{\infty}(\Omega_1)},\Omega_1)>0$ and  
$ C=C(\Omega_1, \overline{x},\mu_1, \epsilon,  \overline{R},  \Lambda_1, \Lambda_2,q,\|h^{-}\|_{L^{\infty}(\Omega_1)}, \|c_+\|_{L^q(\Omega_1)}) > 0 $ such that 
%{\color{red}A completer} 
\[  \Big( \int_{B_{2R}(\overline{x}) \cap \Omega_1} 
\Big( \frac{v_1(x)}{d(x,\partial \Omega_1)} \Big)^{\epsilon} dx \Big)^{1/\epsilon} \leq C\,.\]
This clearly implies, since $\Omega_1 \subset \Omega$, that
\[  \Big( \int_{B_{2R}(\overline{x}) \cap \Omega_1} 
 v_1(x)^{\epsilon} dx \Big)^{1/\epsilon} \leq  C \diam(\Omega) \,.\]

\noindent The Step 2 then follows observing that 
$0 \leq 1 + \mu_1 w_1 = \mu_1 v_1 + \mu_1 z_2 \leq \mu_1 v_1$ and taking into account that $B_{2R}(\overline{x}) \cap \Omega = B_{2R}(\overline{x}) \cap \Omega_1.$

\medbreak

\noindent \textbf{Step 3:} \textit{Conclusion.}
\medbreak

Arguing exactly as in  Step 3 and  4 of Lemma \ref{Steps A}, 
using  Lemma \ref{BLMP} and Step 2, we show the existence of 
$C_{3} = C_3 (\overline{x},\mu_1, \mu_2, R,  \Lambda_1, \Lambda_2,\|h^{-}\|_{L^{\infty}(\Omega_1)}, \|c_+\|_{L^q(B_{2R}(\Omega_1)}) > 0 $
such that
\[ \sup_{B_{R}(\overline{x}) \cap \Omega} w_2 \leq C_{3}\,.\]
Hence, the proof of the lemma follows by the definition of $w_2$. 
%Directly observe that
%We conclude the proof arguing exactly as in  Step 4 of Lemma \ref{Steps A}, 
%using  Lemma \ref{BLMP} and Step 2. 
%Now, arguing exactly as in the Step 3 of Lemma \ref{Steps A}, using in this case Lemma \ref{BLMP} and 
%Step 2, we can conclude that there exists 
%$C_{3} = C_{3} (C_{3}, \mu_2, \Lambda_2, \|c^{+}\|_q,\|h\|_q ) > 0$ such that
%\[ \sup_{B_{R/2}(\overline{x}) \cap \Omega} w_2 \leq C_{3}\,.\] 
\end{proof}

\begin{proof}[\textbf{Proof of Theorem \ref{aPrioriBound}}]

\noindent Arguing by contradiction we assume the existence of 
sequences $\{\lambda_n\}\subset [\Lambda_1 ,\Lambda_2]$,   $\{u_n\}$ 
 solutions of \eqref{Plambda} for $\lambda=\lambda_n$
 % satisfying {\color{red}$u_n \geq u_0$ for all $n \in \N$ 
 and of points 
$\{x_n\} \subset \Omega$ such that 
\begin{equation} \label{ab3}
u_n(x_n) = \max\{u_n(x): x \in \overline{\Omega}\} \rightarrow \infty\,, \quad \textup{ as } n \rightarrow \infty\,.
\end{equation} 
Observe that  Lemma \ref{Step 1} and \eqref{ab3} together imply the existence of 
%a sequence $\{v_n\}$ of non-negative solution of \eqref{Plambda} and 
a sequence of points $y_n \in  \overline{\Omega}_{+}$ such that
\begin{equation} \label{ab8}
u_n(y_n) = \max\{u_n(y): y \in \overline{\Omega}_{+}\} \rightarrow \infty\,, 
\quad \textup{ as } n \rightarrow \infty\,.
\end{equation}
Passing to a subsequence if necessary, we may assume that 
$\lambda_n \rightarrow \overline{\lambda} \in  [\Lambda_1 ,\Lambda_2]$ and 
$y_n \rightarrow \overline{y} \in \overline{\Omega}_{+}$. Now, let us distinguish two cases:
\begin{itemize}
\item If $\overline{y} \in \overline{\Omega}_{+} \cap \Omega$, Lemma \ref{Steps A} shows that we can find 
$R_I > 0$ and $M_I > 0$ such that, if $u \in \HLinfty$ is a solution of 
\eqref{Plambda},  then $\sup_{B_{R_{I}}(\overline y)} u\leq M_{I}$.
This contradicts \eqref{ab8}. %$B_{R_{I}}(\overline{y})\subset \Omega_1$ and,
\item If $\overline{y} \in \overline{\Omega}_{+} \cap \partial \Omega$, Lemma \ref{Steps B} shows 
that we can find $R_B > 0$ and $M_B > 0$ such that, if $u \in \HLinfty$ is a solution of \eqref{Plambda}, 
then $\sup_{B_{R_{B}}(\overline y)\cap \Omega} u\leq  M_B$. Again, this contradicts 
\eqref{ab8}.
\end{itemize}
As  \eqref{ab8} cannot happen,  the result follows.
%Hence, \eqref{ab3} cannot happen and so, we can conclude that there exists a constant $M > 0$ such that
%\[ \|u\|_{\infty} \leq M\,.\]
\end{proof}

\section{Proof of Theorem \ref{th1}} \label{IV}

Let us begin with a preliminary result.

\begin{lemma} \label{nonExistenceLambdaLarge}
Under the assumption \eqref{A1}, assume that  $(P_0)$ has a solution $u_0$ for which there exist $\overline x\in \Omega$ and $R>0$ such that $c_+u_0 \gneqq 0$, $c_-\equiv 0$ and $\mu\geq 0$ in $B_R(\overline x)$. 
Then there exists $\overline\Lambda\in\, (0,\infty)$ 
 such that, for $\lambda \geq \overline\Lambda$, the problem \eqref{Plambda} has no solution $u$ with $u \geq u_0$ in $B_R(\overline x)$.
\end{lemma}

\begin{proof} 
Let us introduce $\overline{c}(x) := \min\{c_{+}(x),1\}$. Observe that $0 \lneqq \overline{c} \leq c_{+}$
%Let $C>0$   be such that the function $\overline c= \min(c_+, C)$ satisfies $\overline c\gneqq 0$. 
and define $\gamma^1_1 > 0$ as the first eigenvalue of the problem
\begin{equation} \label{Pgamma}
\left\{
\begin{aligned}
-\Delta \varphi &  = \gamma \overline c (x)\varphi & \textup{ in } B_R(\overline x),\\
\varphi & = 0 & \textup{ on } \partial B_R(\overline x).
\end{aligned}
\right.
\end{equation}By standard arguments, there exists $\varphi_1^1 \in {\mathcal C}_0^1(\overline {B_R(\overline x)})$  an associated first eigenfunction such that 
$\varphi_1^1(x)>0$ for all $x\in B_R(\overline x)$ and, 
denoting by $n$ the outward normal to $\partial B_R(\overline x)$, we also have
\begin{equation}\label{derive-negative}
\frac{\partial \varphi_1^1(x)}{\partial n} < 0\,, \quad \textup{on } \partial B_R(\overline x).
\end{equation}
Now, let us introduce the function $\phi \in \HLinfty$, defined as
\[ \phi(x) = \left\{ \begin{aligned}
& \varphi_1^1(x)\,, \quad & x \in B_R(\overline x), \\
& 0 & x \in \Omega \setminus B_R(\overline x),
\end{aligned}
\right.
\]
and suppose that $u$ is a solution of \eqref{Plambda} such that $u \geq u_0$ in $B_R(\overline x)$. First observe that, in view of \eqref{derive-negative} and as $u\geq u_0$  on $\overline{B_R(\overline x)}$, there exists a constant  $C>0$ independent of $u$ such that
\begin{equation}\label{derive-negative-bis}
 \int_{\partial B_R(\overline x)} u \frac{\partial \varphi_1^1}{\partial n} \,dS \leq C.
\end{equation}
Thus on one hand, using  \eqref{Pgamma} and \eqref{derive-negative-bis}, we obtain % when $\gamma = \gamma_1^1\,,$ we obtain that
\begin{equation} \label{ne1}
\begin{aligned}
 \int_{\Omega}\big( \nabla \phi \gradu  + c_{-} &(x)\, \phi\, u\big)\, dx 
= 
 \int_{B_R(\overline x)} \nabla \varphi_1^1 \gradu\, dx 
= 
- \int_{B_R(\overline x)} u \Delta \varphi_1^1\, dx 
+ 
\int_{\partial B_R(\overline x)} u \frac{\partial \varphi_1^1}{\partial n} \,dS 
\\
&\leq 
 - \int_{B_R(\overline x)} u \Delta \varphi_1^1 \,dx + C 
 =  \gamma_1^1 \int_{B_R(\overline x)}\overline{c}(x) \,\varphi_1^1 \,u\, dx + C
\leq \gamma_1^1 \int_{\Omega} c_{+}(x) \,\phi\, u\, dx\, + C.
\end{aligned}
\end{equation}
On the other hand,  considering $\phi$ as test function in \eqref{Plambda} we observe that
\begin{equation} \label{ne2}
\int_{\Omega} \big(\nabla \phi \gradu + c_{-}(x) \,\phi \,u\big)\, dx 
= \lambda \int_{\Omega} c_{+}(x) \,u\,\phi \, dx + \int_{\Omega} \big( \mu(x) |\gradu|^2 + h(x)\big) \phi\, dx\,.
\end{equation}
From \eqref{ne1} and \eqref{ne2}, we then deduce that, for a $D>0$ independent of $u$.
\begin{equation} \label{ne3}
\begin{aligned}
 (\gamma_1^1 - \lambda) \int_{\Omega} c_{+}(x) \,\phi \,u\, dx 
& \geq 
 \int_{\Omega} \big( \mu(x) |\gradu|^2 + h(x) \big) \phi\, dx - C
\\
& = 
 \int_{B_R(\overline x)} \big( \mu(x) |\gradu|^2 + h(x) \big) \varphi_1^1\, dx  - C  \geq -D.
\end{aligned}
\end{equation}
 As   $c_+u_0 \gneqq 0$ in $B_R(\overline x)$, we have that 
 \[
 \int_{\Omega} c_{+}(x) \,\phi \,u\, dx 
 \geq \int_{\Omega} c_{+}(x) \,\phi \,u_0\, dx  >0.
 \]
Hence, for $\lambda > \gamma_1^1$ large enough, we obtain a contradiction with \eqref{ne3}.
\end{proof}

%\begin{remark}
%Observe that in case $h \in L^p(\Omega)$ for some $p > N$, $c^{-}$, $\mu \in \Linfty$ and $h \gneqq 0$, we have  $u_0(x) > 0$ for a.e. $x\in \Omega$ and \eqref{ne3} gives a contradiction in case $\lambda \geq \gamma_1^1$. 
%\end{remark}

\begin{proof}[\textbf{Proof of  Theorem \ref{th1}}]  We treat separately the cases $\lambda \leq 0$ and $\lambda > 0$.

\medbreak 

\noindent \textbf{Part 1:} \textit{$\lambda\leq0$.} 
\medbreak

Observe that for $\lambda\leq 0$ we have $\lambda c_+-c_-\leq -c_- $ and hence the result 
follows from \cite[Lemma 5.1, Proposition 4.1, Proposition 5.1, Theorem 2.2]{A_DC_J_T_2015}
as in the proof of  \cite[Theorem 1.2]{A_DC_J_T_2015}.  Moreover, observe that $u_0$ is an upper solution of  \eqref{Plambda}. Hence we conclude that 
$u_{\lambda}\leq u_0$ by 
 \cite[Lemmas 2.1 and 2.2]{A_DC_J_T_2014}.
\medbreak

\noindent \textbf{Part 2:} \textit{$\lambda > 0$.}
\medbreak

Consider, for $\lambda \geq 0$  the modified problem
\[ \label{Pbar} \tag{$\overline{P}_\lambda$}
 -\Delta u +u=  (\lambda c_+(x) - c_-(x)+1) \, ((u-u_0)^{+} +u_0)+\mu(x)|\nabla u|^2 + h(x)\,, \quad u \in \HLinfty.
\]
As in the case of $(P_{\lambda})$, any solution of  
 \eqref{Pbar} belongs to $\mathcal{C}^{0,\tau}(\overline{\Omega})$ for some $\tau > 0$. Moreover, observe that $u$ is a solution of \eqref{Pbar} if and only if it is a fixed point of the operator  %\in \mathcal{C}(\overline{\Omega})
$\overline T_\lambda$ defined by 
  $\overline T_\lambda: {\mathcal C}(\overline\Omega)\to {\mathcal C}(\overline\Omega): v\mapsto u$ 
  with $u$ the solution of
\[ 
 -\Delta u +u-\mu(x)|\nabla u|^2 =  (\lambda c_+(x) - c_-(x)+1)\, ((v-u_0)^{+} +u_0)+h(x)\,, \quad u \in \HLinfty.
\]
Applying \cite[Lemma 5.2]{A_DC_J_T_2015},  we see that $\overline T_\lambda$ is completely continuous. Now, we denote
\[ 
\overline\Sigma := 
\{ (\lambda, u ) \in \R \times \mathcal{C}(\overline{\Omega}) : u \textup{ solves } \eqref{Pbar}  \}\,
\]
%\medbreak
and we split the rest of the proof into three steps.
\medbreak
\noindent \textbf{Step 1:} \textit{If $u$ is a solution of \eqref{Pbar} then $u\geq u_0$ 
and hence it is a solution of \eqref{Plambda}.} 
\medbreak

Observe that $(u-u_0)^{+} +u_0-u\geq  0$. Also we have that $\lambda c_+(x)((u-u_0)^{+} +u_0) \geq \lambda c_+(x)u_0\geq 0$. Hence, we deduce 
that a solution $u$ of  \eqref{Pbar} is an upper solution of 
\begin{equation}
\label{eqThm1.2}
 -\Delta u =  - c_-(x) \, ((u-u_0)^{+} +u_0)+\mu(x)|\nabla u|^2 + h(x)\,, \quad u \in \HLinfty.
\end{equation}
Then the result follows from 
 \cite[Lemmas 2.1 and 2.2]{A_DC_J_T_2014} noting that $u_0$ is a solution of \eqref{eqThm1.2}.
\medbreak

\noindent \textbf{Step 2:} \textit{$u_0$ is the unique solution of  $(\overline{P}_0)$ and 
$i(I-\overline T_0,u_0) = 1$.} 
\medbreak
Again the uniqueness of the solution of $(\overline{P}_0)$ can be deduced from 
\cite[Lemmas 2.1 and 2.2]{A_DC_J_T_2014}. Now, in order to prove that  $i(I-\overline T_0,u_0) = 1$, we consider the operator $S_t$ defined by $S_t: {\mathcal C}(\overline\Omega)\to {\mathcal C}(\overline\Omega): v\mapsto u$ 
with $u$ the solution of
\[ 
 -\Delta u +u-\mu(x)|\nabla u|^2 = t[ (- c_-(x)+1)\, (u_0+ (v-u_0)^{+} -(v-u_0-1)^+)+h(x)]\,, \quad u \in \HLinfty.
\]
First, observe that there exists $R>0$ such that, for all $t\in [0,1]$ and all $v\in {\mathcal C}(\overline\Omega)$, 
\[
\|S_t v\|_{\infty}<R.
\]
This implies that
\[
 \textup{deg}(I-S_1, B(0, R))=
 \textup{deg}(I, B(0, R))=
1.
\]
By  \cite[Lemmas 2.1 and 2.2]{A_DC_J_T_2014}, we see that $u_0$ is the only fixed point of  $S_1$. 
Hence, by the excision property of the degree, for all $\epsilon>0$ small enough, it follows that
\[
 \textup{deg}(I-S_1, B(u_0, \epsilon))=
 \textup{deg}(I-S_1, B(0, R))= 1.
\]
Thus, as for $\epsilon<1$, $S_1=\overline T_0$, we conclude that 
\[
i(I-\overline T_0,u_0)=\lim_{\epsilon\to0}
 \textup{deg}(I-\overline T_0, B(u_0, \epsilon))=\lim_{\epsilon\to0}
 \textup{deg}(I-S_1, B(u_0, \epsilon))= 1.
\]
%\pagebreak

%\medbreak
\noindent \textbf{Step 3:} \textit{Existence and behavior of the continuum.} 
\medbreak

By \cite[Theorem 3.2]{R_1971} (see also \cite[Theorem 2.2]{A_DC_J_T_2015}), there exists a continuum 
$\mathscr{C} \subset \overline\Sigma$ such that 
$\mathscr{C} \cap ( [0,\infty) \times \mathcal{C}(\overline{\Omega}))$
is unbounded. By Step 1, we know that if $u \in \mathscr{C} $ then $u \geq u_0$ and is a solution of $(P_{\lambda})$. Thus applying Lemma \ref{nonExistenceLambdaLarge}, we deduce that 
$\textup{Proj}_{\R}\mathscr{C} \cap [0,\infty)  \subset [0, \overline \Lambda]$. By Theorem \ref{aPrioriBound} and Step 1, we deduce that for every $\Lambda_1 \in (0,\overline\Lambda)$, 
there is an a priori bound on the solutions of \eqref{Pbar} for $\lambda \in [\Lambda_1,\overline\Lambda]$. 
Hence, the projection of 
$\mathscr{C} \cap ( [\Lambda_1, \overline\Lambda)\times \mathcal{C}(\overline{\Omega}))$ on 
$\mathcal{C}(\overline{\Omega})$ is bounded, and so, we deduce that $\mathscr{C}$ emanates from infinity 
to the right of $\lambda = 0$. Finally, since $\mathscr{C}$ contains $(0, u_0)$ with $u_0$ the unique solution 
of $(P_0)$, we conclude that there exists $\lambda_0 \in (0,\overline\Lambda)$ such that problem
 \eqref{Pbar}, and thus problem  $(P_{\lambda})$, has at least two solutions satisfying $u \geq u_0$ for $\lambda \in (0,\lambda_0)$. 
\end{proof}

\section*{Acknowledgements} 
 
The authors thank warmly Prof. David Arcoya for helpful discussions having lead to improvements of the results. 
Part of this work was done during the visit of the first author to the University of Bourgogne-Franche-Comt\'e. She thanks the LMB for its hospitality and the R\'egion Bourgogne-Franche-Comt\'e for the financial support.
\bigbreak

\bibliographystyle{plain}
\bibliography{Bibliography}

\end{document}